\documentclass{amsart}
\usepackage{amssymb,amsfonts}
\usepackage[all,arc]{xy}
\usepackage{enumerate}
\usepackage{color,soul}
    \sethlcolor{white}
\usepackage{mathrsfs}
\usepackage{pinlabel}
\usepackage{anyfontsize}
\usepackage{t1enc}
\usepackage{rotating}
\usepackage{xparse}
\usepackage{graphicx}
\usepackage{caption}
\usepackage{subcaption}
\usepackage{pinlabel}
\usepackage{subcaption,graphicx}
\newtheorem{thm}{Theorem}[section]

\newtheorem{cor}[thm]{Corollary}
\newtheorem{prop}[thm]{Proposition}
\newtheorem{lem}[thm]{Lemma}

\usepackage{graphicx}
\theoremstyle{definition}
\newtheorem{defn}[thm]{Definition}

\newtheorem{exmp}[thm]{Example}

\theoremstyle{remark}

\makeatletter
\let\c@equation\c@thm
\makeatother
\numberwithin{equation}{section}

\title{Graphoids}

\author{Neslihan G\"{u}g\"{u}mc\"{u}, Louis H.Kauffman, and  Puttipong Pongtanapaisan}

\begin{document}

 \begin{abstract}
We study invariants of virtual graphoids, which are virtual spatial graph diagrams with two distinguished degree-one vertices modulo graph Reidemeister moves applied away from the distinguished vertices. Generalizing previously known results, we give topological interpretations of graphoids. There are several applications to virtual graphoid theory. First, virtual graphoids are suitable objects for studying knotted graphs with open ends arising in proteins. Second, a virtual graphoid can be thought of as a way to represent a virtual spatial graph without using as many crossings, which can be advantageous for computing invariants. 
 \end{abstract}

\maketitle
\section{Introduction}
Motivated by the fact that it is not necessary to draw the entire knot diagram to specify a knot, Turaev introduced the notion of knotoids \cite{Turaev}. Recall that a knotoid is an equivalence class of immersions of an interval into the plane or the sphere modulo the classical Reidemeister moves applied away from the endpoints of the interval. Since knotted proteins are open curves, knotoids function as suitable models to study protein entanglements \cite{Dab,Dim}. It is now known that  there is a one-to-one correspondence between spherical knotoids and ordered oriented $\theta$-curves in $S^3$ with a preferred unknotted cycle \cite{Turaev}, and that there is a one-to-one correspondence between planar knotoids and line-isotopy classes of smooth open oriented curves in 3-space \cite{GK}. Furthermore, Kauffman's virtual knot theory \cite{KVirtual} turns out to provide useful bounds for knotoid invariants \cite{GK}.

 Another topological structure that arises in proteins is a spatial graph \cite{Dab}. Our goal is to introduce virtual graphoids, study their invariants, and we hope to use it to study open spatial graphs in proteins. Our primary way of obtaining invariants for a graphoid is to connect up the open ends in the diagram by an arc, declare any crossing that is a consequent of that arc intersecting the diagram to be all virtual crossings, and then study the invariant of the associated virtual spatial graphs. Some invariants of classical graphoids were studied in \cite{Graphoid} and \cite{Adams}. In this paper, we generalize those invariants to the virtual setting and investigate some new invariants such as the Yamada polynomial. Our crossing number analysis is inspired by methods used by Motohashi, Ohyama, and Taniyama in \cite{Motohashi}.

This paper is organized as follows. In Section \ref{sec:closures}, we define virtual graphoids and discuss ways to associate other knotted objects to spherical graphoids by connecting up the open ends in the diagram. Namely, taking the underpass closure of a classical graphoid gives a classical spatial graph and taking the virtual closure to a virtual graphoid gives a virtual spatial graph associated to the graphoid. We demonstrate that graphoid gives a simpler presentation of a classical spatial graph, so that calculating invariants such as the fundamental group can be done faster on a graphoid diagram.

In Section \ref{sec:TopologicalInterpretation}, we discuss topological interpretations of topological graphoids (planar and spherical). In Section \ref{section:replacement}, we gather simple ways to obtain invariants of virtual graphoids by performing local replacements in the diagrams. In Section \ref{section:Yamada} and \ref{section:pure}, we discuss Yamada polynomials for virtual graphoids and discuss its application to the crossing number.
\section*{Acknowledgement}
We would like to thank Eleni Panagiotou for posing the question that inspired this paper at the BIRS workshop 21w5232. We are grateful for helpful discussions and encouragement from Kasturi Barkataki, Micah Chrisman, and Homayun Karimi. Research conducted for this paper is supported by the Pacific Institute for the Mathematical Sciences (PIMS). The research and findings may not reflect those of the Institute.
\section{Graphoids and their closures}\label{sec:closures}
A \textit{virtual graphoid diagram} is a graph with two distinguished degree one vertices generically immersed in $S^2$, where each double point is decorated as either a classical crossing or a virtual crossing. We call the two distinguished degree-one vertices the \textit{head}, and the \textit{tail} (some authors referred to the tail as the \textit{leg)}.

A (topological) \textit{virtual graphoid} is an equivalence class of virtual graphoid diagrams modulo Reidemeister moves shown in Figure \ref{graphReid}, Figure \ref{fig:virtgraphReid} and isotopy of the plane.  If we do not allow move (VI) in our equivalence relation, then we obtain a \textit{rigid vertex virtual graphoid}. Note that move (V) in Figure \ref{graphReid} is the rigid vertex move. The moves shown in Figures \ref{forbidden} and \ref{fig:forbiddengraph} are not allowed. Note that in move (IV) the vertex through which we allow a strand with crossings on it can slide, is of some degree $k \geq 2$, that is, it is not the head or the tail of a graphoid diagram.

We will not distinguish between an edge and an edge containing degree 2 vertices. Note that forbidding the move in Figure \ref{fig:virtforbidden} produces a strict version of virtual knotoids that we do not use in this paper. Thus, the reader can assume for the present work that Figure 4 move is allowed. This means that we allow general detour moves: a consecutive sequence of virtual crossings can be excised and replaced by any arc between the same endpoints,that also has a consecutive sequence of virtual crossings.

We remark that in the definition above, if we do not allow virtual crossings in our diagram and we only consider our diagrams up to the Reidemeister moves in Figure \ref{graphReid}, we obtain \textit{classical graphoids}, which was also defined in \cite{Graphoid}. 

\begin{defn}
    We say that a classical graphoid is \textbf{pure} if it does not admit a diagram where the head and the tail lie in the same region of the diagram.
\end{defn}
\begin{figure}[ht!]

\centering
\includegraphics[width=8cm]{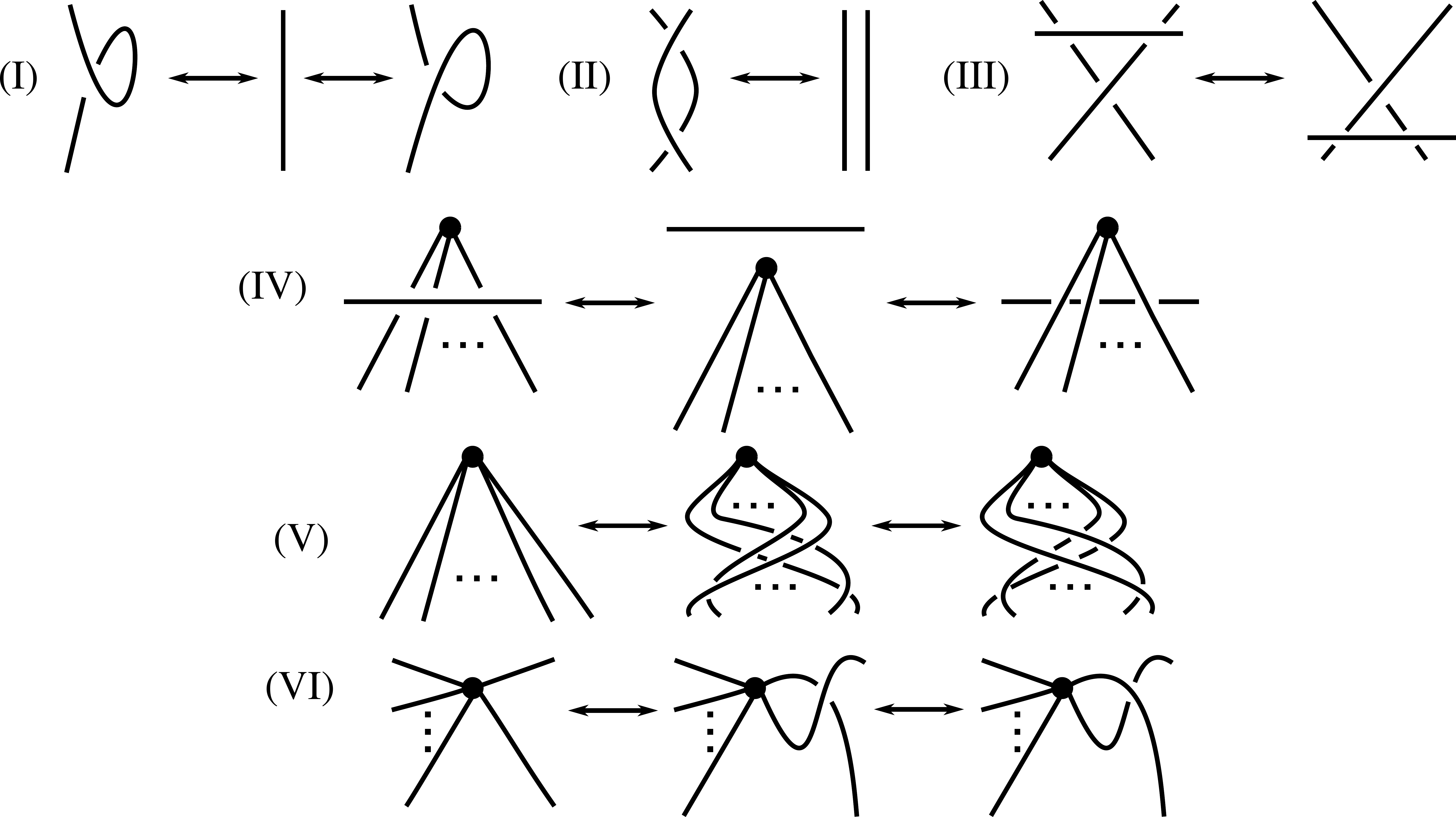}
\caption{Classical Reidemeister moves for spatial graphs.}\label{graphReid}
\end{figure}
\begin{figure}[ht!]

\centering
\includegraphics[width=7cm]{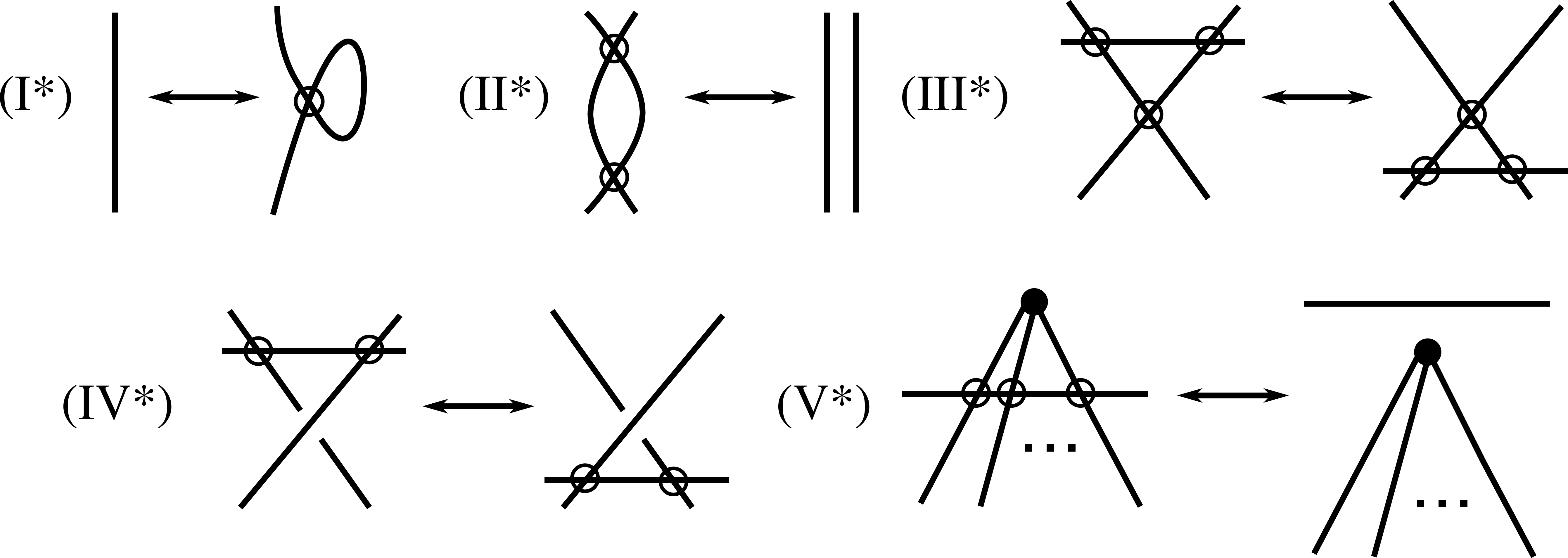}
\caption{Reidemeister moves involving virtual crossings.}\label{fig:virtgraphReid}
\end{figure}
\begin{figure}[ht!]

\centering
\includegraphics[width=4cm]{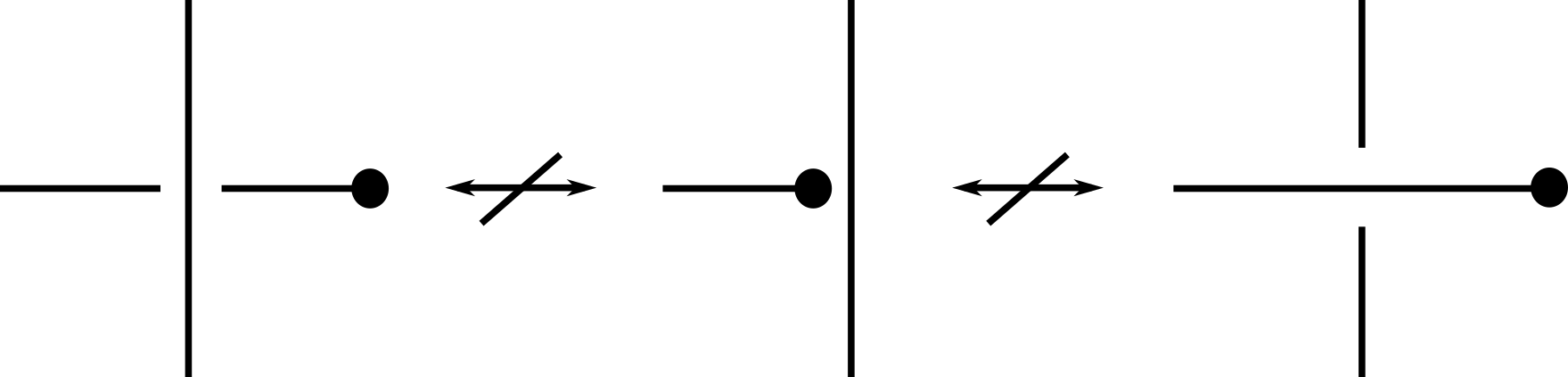}
\caption{A forbidden move involving a classical crossing and an open end.}\label{forbidden}
\end{figure}

\begin{figure}[ht!]

\centering
\includegraphics[width=4cm]{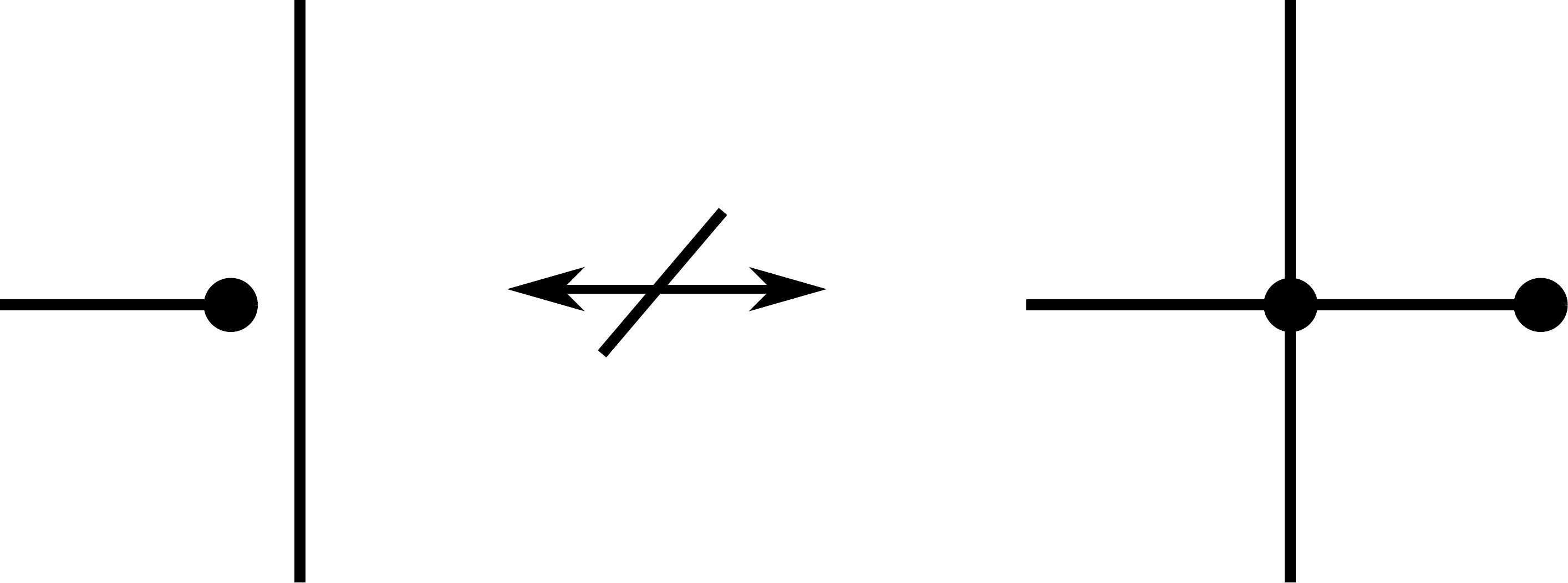}
\caption{A forbidden move involving a degree 4 vertex and an open end.}\label{leaftuck}
\end{figure}
\begin{figure}[ht!]
\centering
\includegraphics[width=3cm]{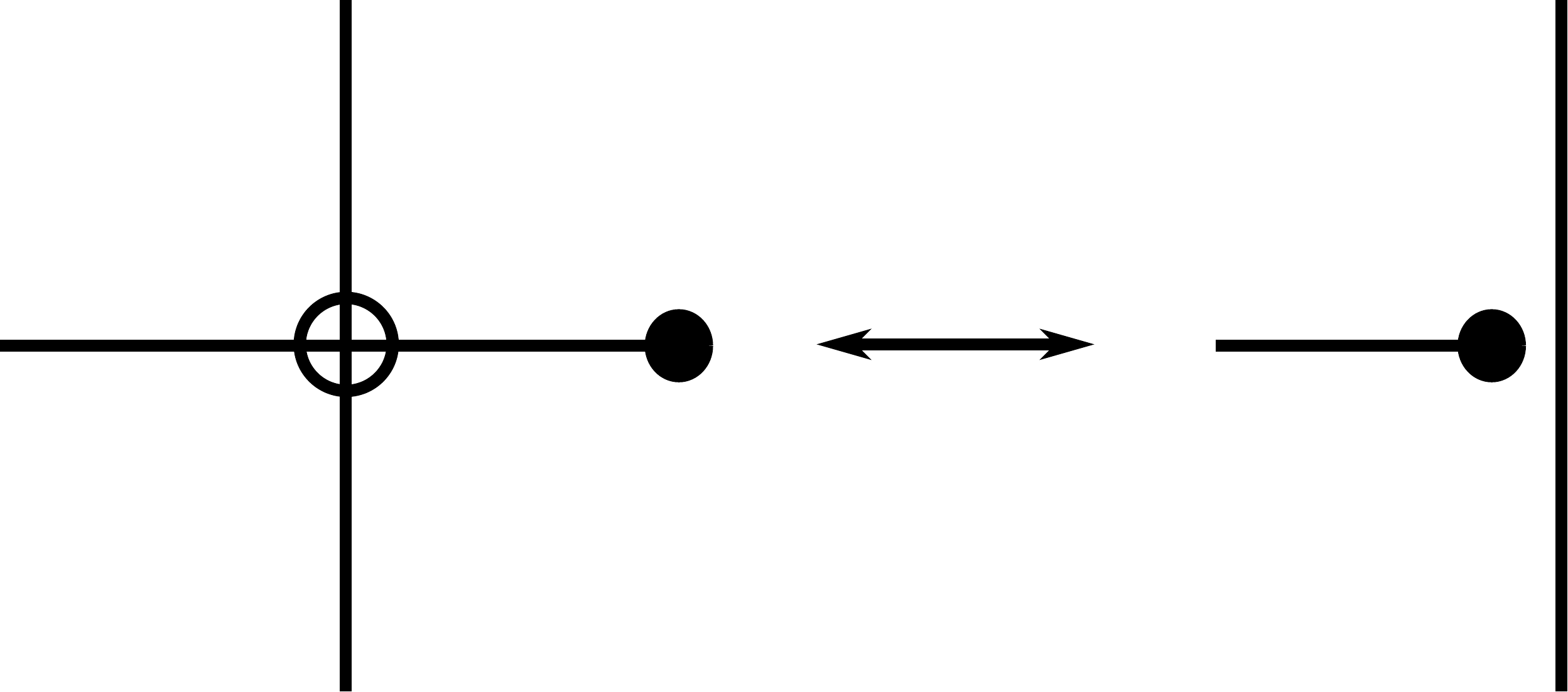}
\caption{A move involving a virtual crossing and an open end.}\label{fig:virtforbidden}
\end{figure}
\begin{figure}[ht!]

\centering
\includegraphics[width=4cm]{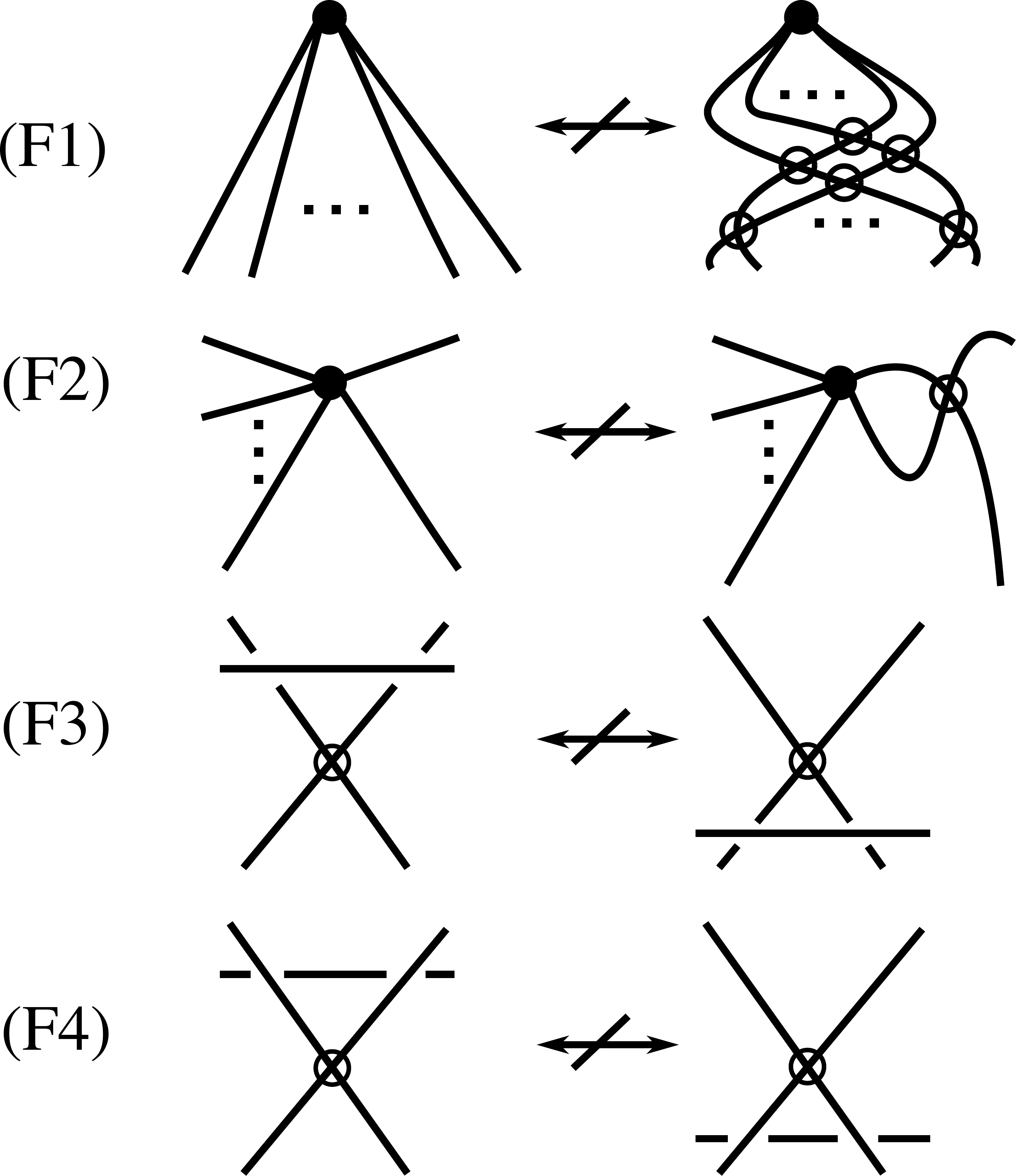}
\caption{A list of forbidden moves.}\label{fig:forbiddengraph}
\end{figure}

We now discuss a few ways of connecting up the head and the tail of our virtual graphoid diagram so that virtual spatial graph theory can be applied. 

\subsection{Classical graphoids and spatial graphs}
An \textit{overpass closure} (resp. an \textit{underpass closure}) of a classical graphoid $G$ is the classical spatial graph obtained by connecting the head to the tail with an embedded arc, called a \textit{shortcut}, that goes over (resp. under) each strand it meets during the connection. Obviously, any two shortcuts for $G$ are isotopic, which implies that there is a well-defined map $\omega_+$ (resp. $\omega_-$) from the set of classical graphoids to the set of spatial graphs in $\mathbb{R}^3$. However, one spatial graph $\Gamma$ can be the image of two non-equivalent classical graphoids. One may obtain explicit examples of this phenomenon by removing two distinct arcs from a diagram $\Gamma$ at different locations. Therefore, the map $\omega_+$ (resp. $\omega_-$) is surjective, but non-injective. Observe that the overpass closure and the underpass closure of a classical graphoid diagram may give rise to non-isotopic spatial graphs. When we consider classical graphoids, we will consistently focus on the underpass closure throughout the paper.  

As advertised, we demonstrate through the fundamental group computation that representing a spatial graph with a classical graphoid diagram may make calculations of spatial graph invariants simpler. Notice that we can perform the Wirtinger algorithm on a graphoid diagram in a natural way and we denote the resulting group by $\pi_1(G)$.

\begin{lem}
$\pi_1(G)$ is isomorphic to the fundamental group of $\mathbb{R}^3\backslash \omega_-(G)$.
\end{lem}
\begin{proof}
Consider one of the distinguished degree-one vertex $v$, say the head of a graphoid. First, we claim that the forbidden move $\Omega_-$ that slides $v$ under another strand does not change $\pi_1(G)$. Suppose that $\langle x_1,\ldots,x_n | R \rangle$ is a presentation before the $\Omega_-$ move is performed. After the $\Omega_-$ move is performed, we get a presentation $\langle x_1,\ldots,x_n,y | R,y=x_j^{\pm 1}x_ix_j^{\mp 1} \rangle$, where $x_i$ and $x_j$ are generators associated to the strands involved in the $\Omega_-$ move. These two group presentations are equivalent by a Tietze transformation.

It follows that $\pi_1(G)$ is isomorphic to $\pi_1(G')$, where $G'$ is obtained from $G$ by performing a sequence of $\Omega_-$ moves until the head and the tail lie in the same region of the diagram. In other words, there is an arc $a$ connecting the head and tail of $G'$ such that $a$ is disjoint from the rest of the diagram. Let $g,h$ be the elements associated to the head and the tail of $G'$, respectively. Then, $\pi_1(\mathbb{R}^3\backslash\omega_-(G))$ has a presentation that is identical to $\pi_1(G')$, but with an additional relation $g=h$. To finish the proof, we show that $g=h$ in $\pi_1(G')$ already. Indeed, $\pi_1(G')$ is isomorphic to the fundamental group of the complement of a graph in a 3-ball $B$ where the graph intersects $\partial B$ in two degree-one vertices. A loop represents $g$ is equivalent to $h$ by wrapping around the sphere $\partial B$.

\end{proof}

\subsection{Virtual graphoids and the virtual closure}
We remark that when virtual crossings are present, the underpass closure is not well-defined. More precisely, one may get distinct virtual spatial graphs if one connects the head to the tail via a shortcut on distinct sides of a virtual crossing because the resulting virtual spatial graphs differ by a forbidden move (see Figure 20 in \cite{GK}).

However, we can turn to a different type of closure. A \textit{virtual closure} is the virtual spatial graph obtained by connecting the head to the tail of a virtual graphoid $G$ with an embedded arc $a$ such that each intersection of $a$ with $G$ is a virtual crossing. Any two ways of connecting up the open ends in this manner are equivalent due to the detour move. Therefore, the virtual closure map is well-defined, but there are examples showing that if we restrict to the set of classical knotoids the virtual closure map it is neither injective nor surjective \cite{GK,GK2}. However, if we extend to the map from the set of virtual graphoids to the set of virtual spatial graphs, then the virtual closure is surjective, but not injective (see Remark (3) in page 205 of \cite{GK}).

\section{Interpretations of Graphoids}\label{sec:TopologicalInterpretation}
In this section, we give geometric interpretations of topological graphoids. In Subsections \ref{subsection:rail}, \ref{subsection:Turaev}, and \ref{Subsection:Buck}, we deal with classical graphoids. Virtual graphoids are considered in Subsection \ref{subsection:virtual}.
\subsection{Planar classical graphoids}\label{subsection:rail}
Let $D$ be a classical graphoid diagram in the plane. Let $h\times \mathbb{R},t\times\mathbb{R}$ be two lines perpendicular to the plane containing $D$ passing through the head and the tail, respectively. Following \cite{GK}, two smooth oriented open ended graphs embedded in $\mathbb{R}^3$ with the endpoints attached to two distinguished lines are said to be \textit{line isotopic} if there is a smooth ambient isotopy of the pair ($\mathbb{R}^3\backslash\{t\times\mathbb{R},h\times\mathbb{R}\}, t\times\mathbb{R},h\times\mathbb{R})$, taking one graph to the other graph in the complement of the lines, taking endpoints to endpoints, and taking lines to lines; $t \times \mathbb{R}$ to $t \times \mathbb{R}$ and $h\times \mathbb{R}$ to $h\times\mathbb{R}$.
We extend a result in \cite{GK} to planar graphoids. Recall that combinatorial isotopy is generated by a \textit{triangle move}, which is defined as follows. To perform a triangle move, one finds a triangle embedded piecewise-linearly in 3-space with interior disjoint from the graph such that the triangle shares one or two edges with the graph. Then, one replaces the edge(s) of the triangle shared with the graph with the unshared edge(s).
\begin{figure}[ht!]

\centering
\includegraphics[width=5cm]{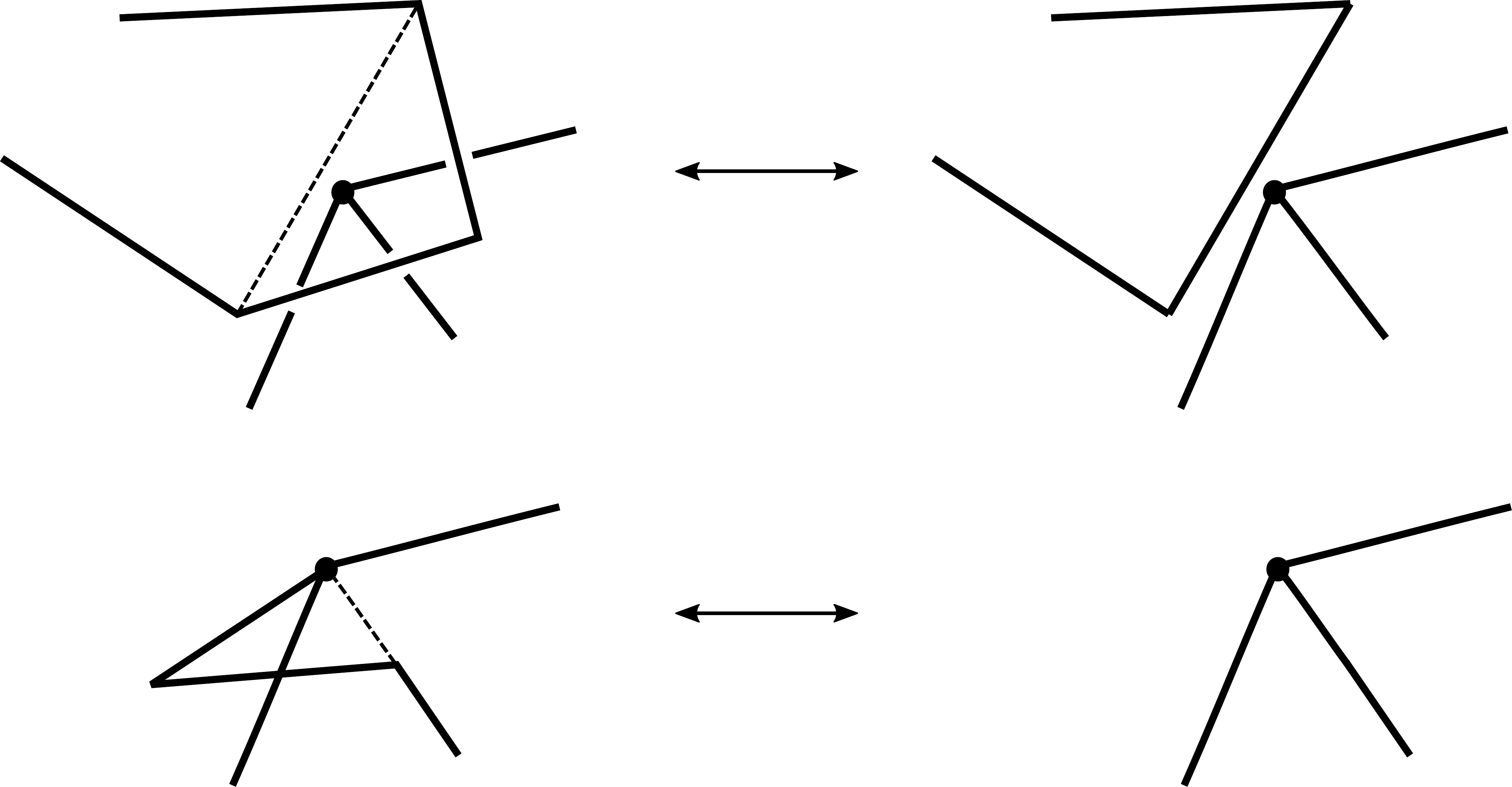}
\caption{Some instances of a triangle move.}\label{fig:plreidemeister}
\end{figure}

\begin{thm}
Let $\Gamma_1$ and $\Gamma_2$ be two smooth oriented spatial graphs with open ends in $\mathbb{R}^3$ that are generic with respect to the $xy$-plane. Then, $\Gamma_1$ and $\Gamma_2$ are line isotopic with respect to the lines passing through the endpoints if and only if their generic projections to the $xy$-plane along the lines are equivalent graphoid diagrams. \label{thm:lineisotopy}
\end{thm}
\begin{proof}
The proof is inspired by Theorem 2.1 of \cite{KGraph}, where we notice that the existence of open ends does not alter the argument. Indeed, we may treat $\Gamma_i$ as a piecewise-linear graph. In other words, $\Gamma_i$ is made up of a finite union of straight edges $[p_1,p_2],\cdots , [p_{n-1},p_n]$, where $p_1$ and $p_n$ are the head and the tail, respectively. It is well known that if $\Gamma_1$ and $\Gamma_2$ are isotopic, then they differ by a finite sequence of triangle moves. By subdividing an edge of a triangle as necessary, the result follows by projecting the triangle defining a triangle move to the plane. Such a triangle will be nonsingular and may contain edges and vertices coming from projecting the piecewise linear curve. One can then consider all the ways that a triangle overlaps with other strands and vertices in the plane and realize that each case is an instance of a Reidemeister move. Some cases of the projections are depicted in Figure \ref{fig:plreidemeister}. The remaining cases are left to the reader.
\end{proof}
\subsection{Classical graphoids on $S^2$ as spatial graphs with an unknot constituent}\label{subsection:Turaev}
A result due to Turaev \cite{Turaev} allows us to regard an oriented classical knotoid on a 2-sphere $S^2$ as an isotopy class of labelled $\theta$-graph in $S^3$ with a preferred unknotted constituent. For this paper, when we have a spatial graph with a preferred unknotted cycle as two edges meeting in two points $v_0$ and $v_1$, we label the edges $e_+$ and $e_-.$ We also label the subgraph $\Gamma \backslash e_+ \cup e_-$ as $\gamma_0$. By an \textit{isotopy of a labelled spatial graph}, we mean an ambient isotopy preserving the labels of these vertices and edges.

The following theorem is a natural generalization of Turaev's theorem to the setting of graphoids.
\begin{thm}\label{thm:turaevanalog}
The set of oriented classical graphoids on $S^2$ is in bijective correspondence with the set of isotopy classes of labelled spatial graphs with a preferred unknotted cycle.
\end{thm}
\begin{proof}
We view $S^3$ as $\mathbb{R}^3 \cup \lbrace \infty \rbrace$. Given a graphoid $G$ on $S^2$ with endpoints $v_0$ and $v_1$. Pick an arc $a$ connecting $v_0$ and $v_1$ in $S^2$. Then, two push offs $e_+$ and $e_-$ of $a$ into the upper half space and into the lower half space, respectively, form an unknotted constituent of a spatial graph $\Gamma$. Reidemeister moves on a graphoid diagram are isotopies of $\Gamma$ and since the diagram of $G$ is on a 2-sphere, any two choices of such arc $a$ give rise to isotopic spatial graphs in $S^3.$

Let $\Gamma$ be a spatial graph in $S^3 = \mathbb{R}^3 \cup \lbrace \infty \rbrace$ with a preferred unknotted bigon $U = e_+ \cup_{(v_0,v_1)} e_-,$ where $e_+$ and $e_-$ project bijectively to an arc $a$ in the equatorial 2-sphere $\mathbb{R}^2 \times \lbrace 0 \rbrace$. The circle $U$ bounds a disk $D,$ and we may isotope such a disk so that $v_0 \cup v_1$ lie in $\mathbb{R}^2 \times \lbrace 0 \rbrace$, $e_-\subset \mathbb{R}^3_-$ and $e_+ \subset \mathbb{R}^3_+$. Assume that $\gamma_0 = \Gamma \backslash e_+ \cup e_-$ intersects $a\times \mathbb{R}$ transversely in finite number of points. Any such intersection points $\gamma_0 \cap a \times \mathbb{R}$ that lies outside of $D$ may be removed by sliding in the horizontal direction following $e_+$ or $e_-$ across $v_0\times \mathbb{R}$ or $v_1 \times \mathbb{R}.$ Since we are considering the graphoid diagram on a 2-sphere, sliding by following $e_+$ is equivalent to sliding by following $e_-$ due to the fact that one can wrap the graphoid diagram around the 2-sphere as shown in Figure \ref{wraparound}. This process gives a classical graphoid diagram in $S^2$ from $\Gamma$ in a well-defined manner.
\end{proof}

\begin{figure}[ht!]

\centering
\includegraphics[width=8cm]{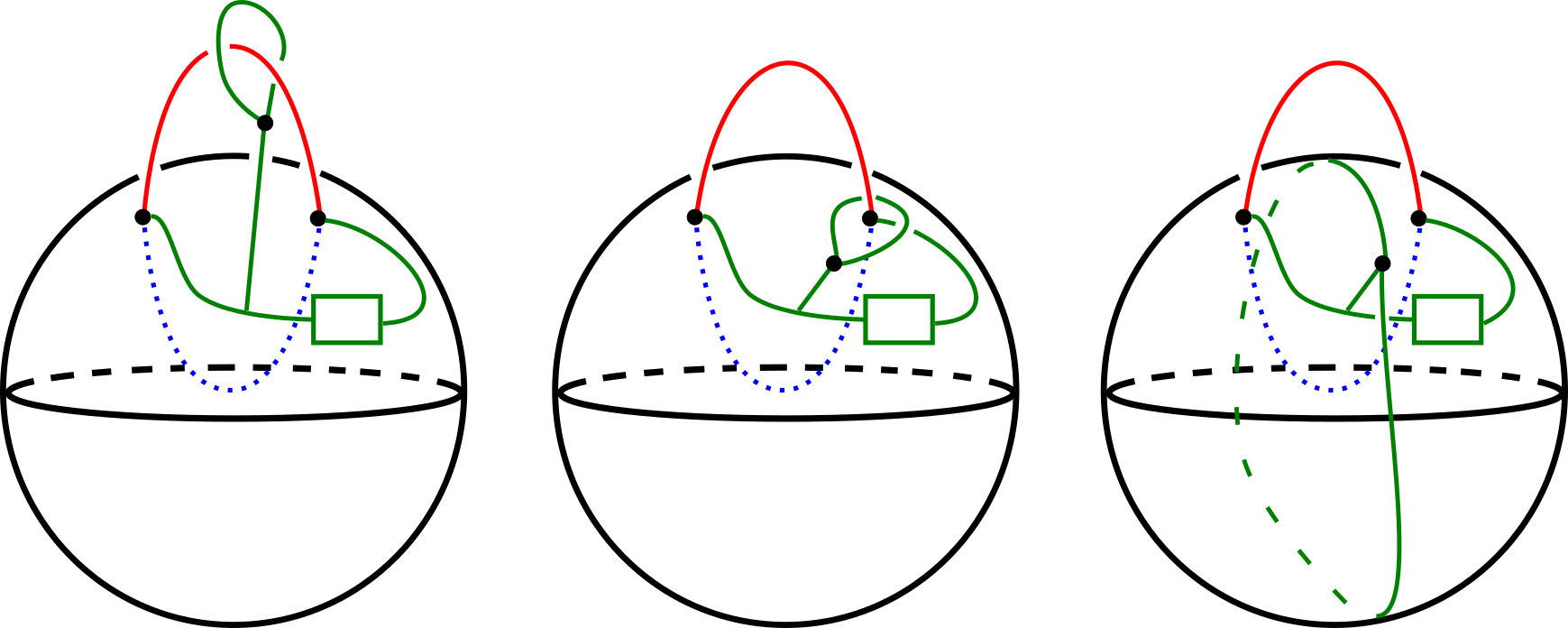}
\caption{The process of wrapping around the sphere mentioned in the proof of Theorem \ref{thm:turaevanalog}.}\label{wraparound}
\end{figure}

A graphoid embedded in $\mathbb{R}^3$ can be also extended to have two parallel lines passing through its endpoints and intersecting the graphoid only at the endpoint vertices. See Figure \ref{fig:longtrefoil} for an example.The resulting object is a \textit{long graphoid}(or \textit{rail graphoid}). The notion of \textit{rail arc}, that is an arc embedded in $\mathbb{R}^3$ with two lines passing through its endpoints,  was introduced in \cite{KodLam} and a long graphoid naturally extends this notion. The \textit{rail isotopy} extends to an isotopy of long graphoids that allows continuous deformations of long graphoids keeping the endpoints on the corresponding parallel lines. A long graphoid $\lambda$ can be projected to a plane that is parallel to the plane spanned by the pair of lines passing through its endpoints. In this way, we obtain the long graphoid diagram corresponding to $\lambda$. The extended Reidemeister moves given for rail knotoid diagrams in \cite{KodLam} can be defined on long graphoid diagrams, and they induce an equivalence relation for them with the graphoid Reidemeister moves (IV), (V), (VI) given in Figure \ref{graphReid}.

The following theorem is a direct generalization of the rail isotopy classes of rail arcs and equivalence classes of rail knotoids correspondence, appearing in \cite{KodLam}.

\begin{thm}
Two long graphoids are rail isotopic in $\mathbb{R}^3$ if and only if their long graphoid diagrams are equivalent.
\end{thm}
\begin{proof} 
The proof follows similarly with the proof of Theorem 3.1. The details are left to the reader.

\end{proof}

\subsection{Classical graphoids in $S^2$ as spatial graphs with order 2 symmetry}\label{Subsection:Buck}
Due to a result of Buck et al. \cite{Buck}, knotoids modulo \textit{rotation}, which is a reflection of $S^2$ containing the diagram about a line through the leg and the head followed by mirror reflection, corresponds to strongly invertible knots up to conjugacy. Here, recall that up to conjugacy means that two strongly invertible knots $(K_1, \tau_1)$ and $(K_2, \tau_2)$ are equivalent if there is an orientation preserving homeomorphism $f: S^3 \rightarrow S^3$ sending $K_1$ to $K_2$ satisfying $f\tau_1f^{-1} = \tau_2$. The following is the analogue for classical graphoids.

\begin{thm}\label{thm:buckanalog}
The set of graphoids on $S^2$ up to rotation is in bijective correspondence with the set of isotopy classes of strongly invertible spatial graphs up to conjugation.
\end{thm}
\begin{proof}
Throughout the proof, we let $t$ denote a bijection between the set of graphoids $G$ and the set of spatial graphs $\Gamma$ with a preferred unknotted constituent $U$, we see that there is a 2-fold cover $f:S^3\rightarrow S^3$ branched along $U$, where $f^{-1}(G)$ is a spatial graph in $S^3$ admitting a strong inversion. Observe that if we reflect the sphere of the diagram $D$ of $G$ about a line through $v_0$ and $v_1,$ we get another graphoid diagram $D'$. Switching all the crossings of $D'$, we get a different graphoid diagram that maps to the same graph admitting a strong inversion corresponding to $G$. This gives a well-defined map from a set of graphoids up to rotation and a set of strongly invertible graphs.

Conversely, suppose that $\Gamma$ is a strongly invertible graph with associated involution $\tau$. Then $f:S^3 \rightarrow S^3/ \tau$ is a 2-fold branched covering. Furthermore, the branched set is the unknot since the Smith Conjecture holds. The branched set together with $f(\Gamma)$ is a spatial graph in $S^3$ containing the branched set as a preferred unknotted constituent. Depending on the four choices of labeling $e_-,e_+,v_0,$ and $v_1$, such a spatial graph may give rise to possibly non-unique, but rotationally equivalent graphoids.
\end{proof}

\begin{figure}[ht!]

\centering
\includegraphics[width=7cm]{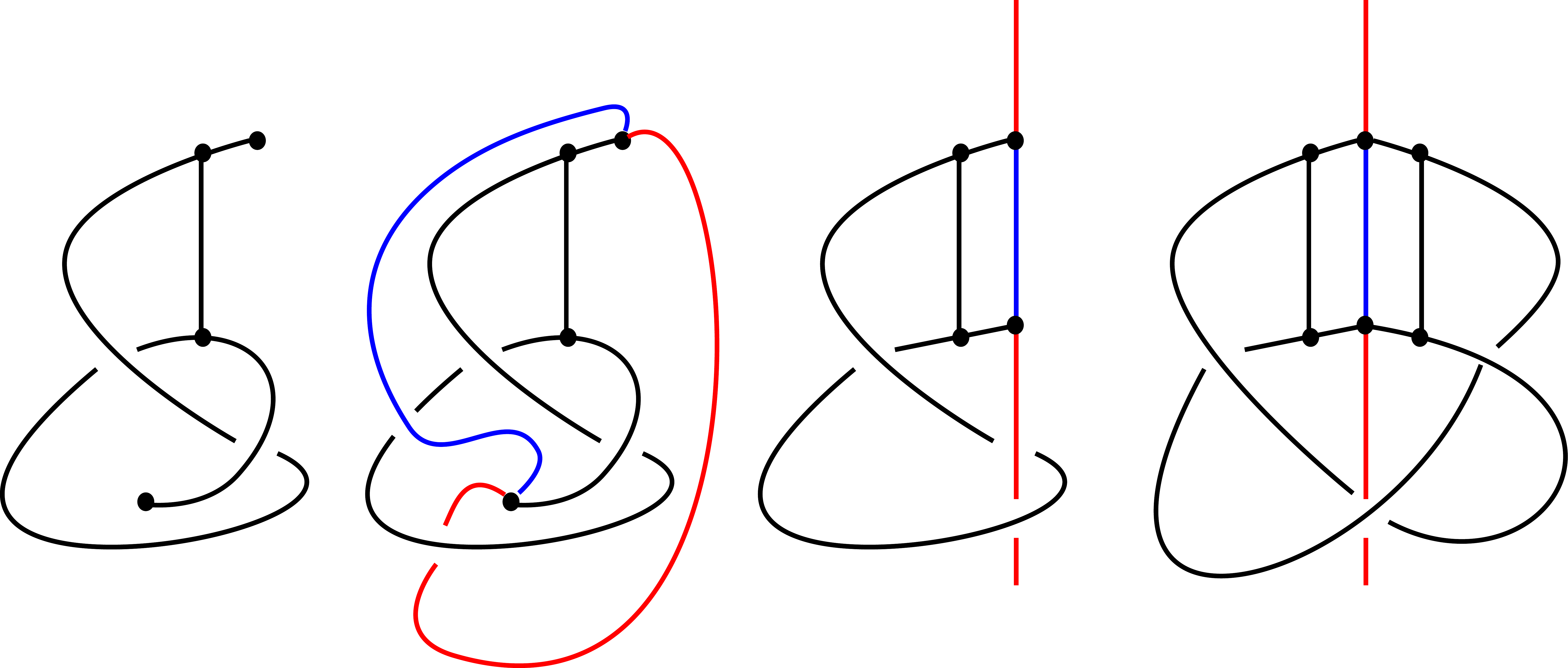}
\caption{From left to right, the first image is a classical graphoid, and the second image is the corresponding spatial graph with a preferred unknotted cycle. The third image is the same as the second image, but the graph is manipulated so that the symmetry is present. In the fourth image, the corresponding strongly invertible spatial graph via the correspondence in Theorem \ref{thm:buckanalog}.}\label{branched}
\end{figure}

The double branched cover method was also used effectively to distinguish planar knotoids. The idea extends in a straightforward manner to classical planar graphoids. More specifically, one can think of a planar graphoid as embedded in $D^2\times I$. Cut $D^2\times I$ along two disks, take two copies, and glue them together. Via this procedure, one obtains a spatial graph in a solid torus as the lift of the planar graphoid, whose isotopy class is an invariant of the graphoid. 
\begin{exmp}
    Consider the graphoids $G_1$ and $G_2$ in Figure \ref{fig:branchedcut}, which is inspired by Figure 7.1 of \cite{Buck}. The associated spatial graphs in solid torus are not equivalent because $G_1$ contains a subknot $K_1$ that is shown in \cite{Buck} to be inequivalent to a subknot $K_2\subset G_2$.
\end{exmp}
\begin{figure}[ht!]

\centering
\includegraphics[width=7cm]{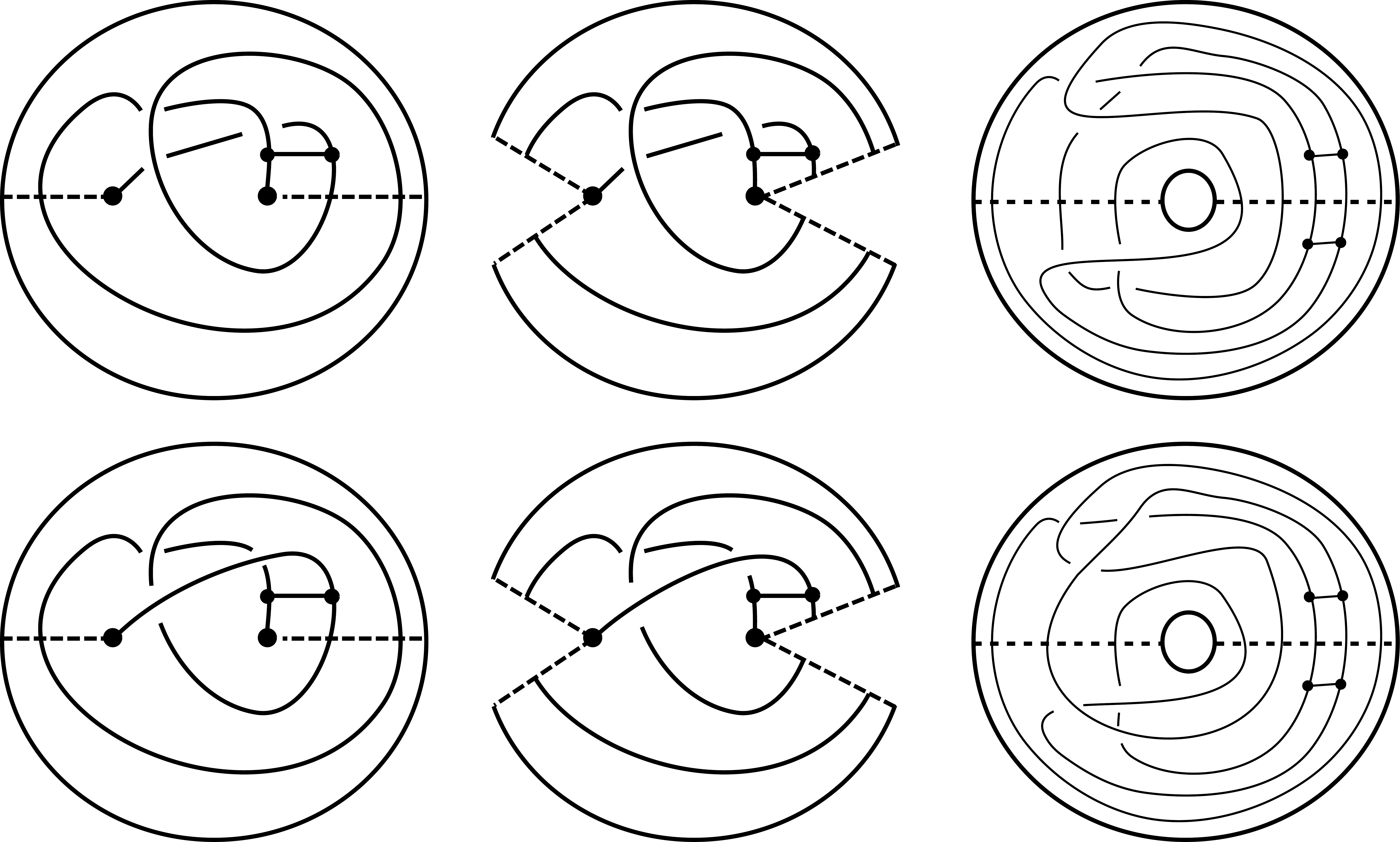}
\caption{Top row: A planar classical graphoid $G_1$. Bottom row: A planar classical graphoid $G_2$.}\label{fig:branchedcut}
\end{figure}

For other existing applications of branched covering technique to spatial graphs' chirality, the readers can consult \cite{Simon}.
\subsection{Virtual graphoids and graphoids in thickened surfaces}\label{subsection:virtual}
It has been established that a virtual spatial graph can be interpreted as a spatial graph in higher
genus surfaces considered up to Reidemeister moves in the surfaces, isotopy of the surfaces and addition and removal of
handles in the complement of spatial graph diagrams. We seek to show a similar statement for virtual graphoids. 

An \textit{abstract graphoid diagram} is a pair $(F,G)$ where $G$ is a graphoid contained in a surface $F$ which is a ribbon neighborhood of $G$ constructed as follows. Each classical crossing lies on a 2-disk. Each vertex also lies on a 2-disk. These 2-disks are then connected by ribbon bands following the strands on the graphoid diagram. At a virtual crossing, one ribbon band passes behind another (see Figure \ref{fig:projectvirtual}). We can then define the abstract Reidemeister moves to be the ribbon versions of the generalized Reidemeister moves for virtual spatial graphs. See Figure \ref{fig:abstractReid} for the abstract Reidemeister A4 move. Note that these abstract Reidemeister moves may alter the genus of $F.$

\begin{figure}[ht!]

\centering
\includegraphics[width=5cm]{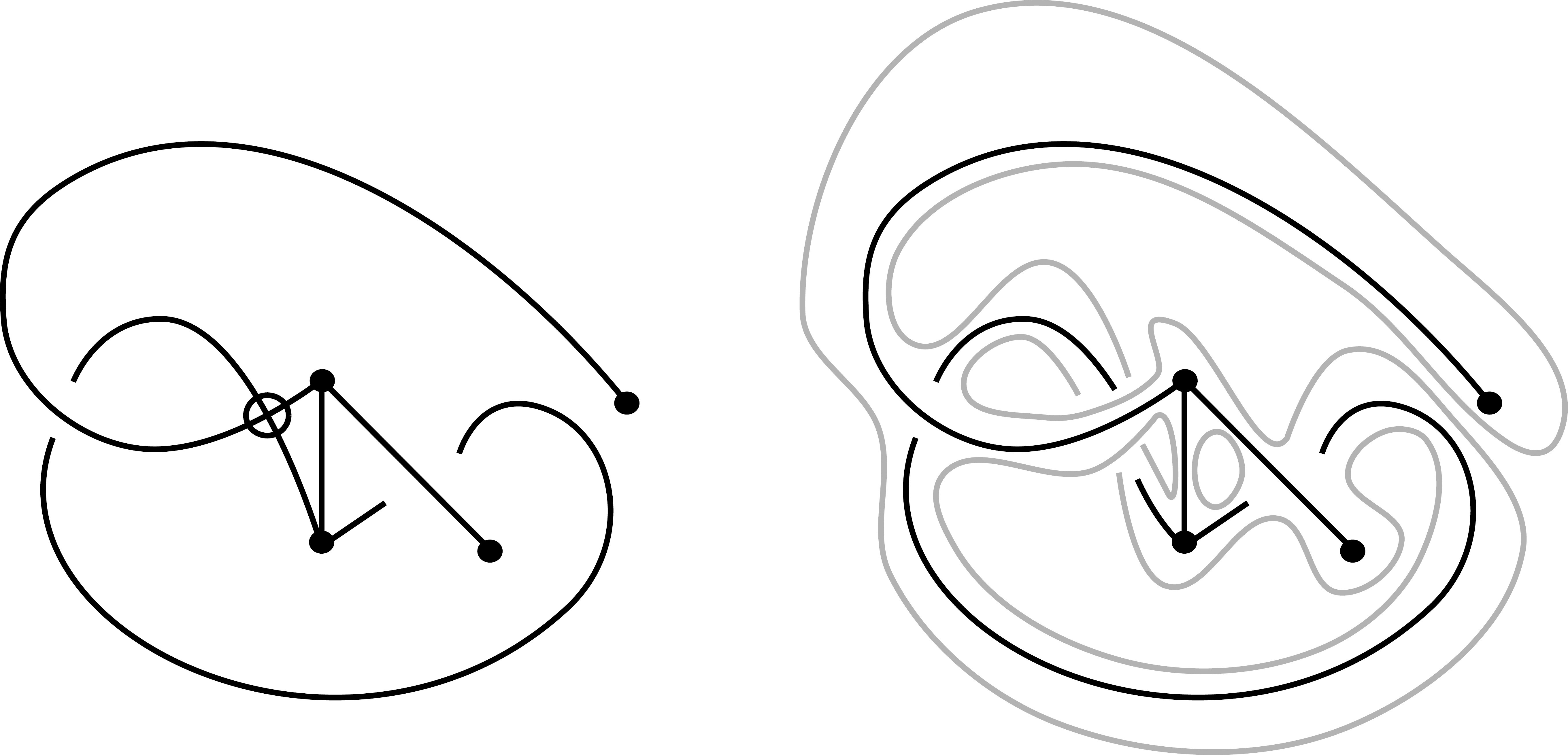}
\caption{A virtual graphoid diagram (left) and the corresponding abstract graphoid diagram (right).}\label{fig:projectvirtual}
\end{figure}

\begin{figure}[ht!]

\centering
\includegraphics[width=5cm]{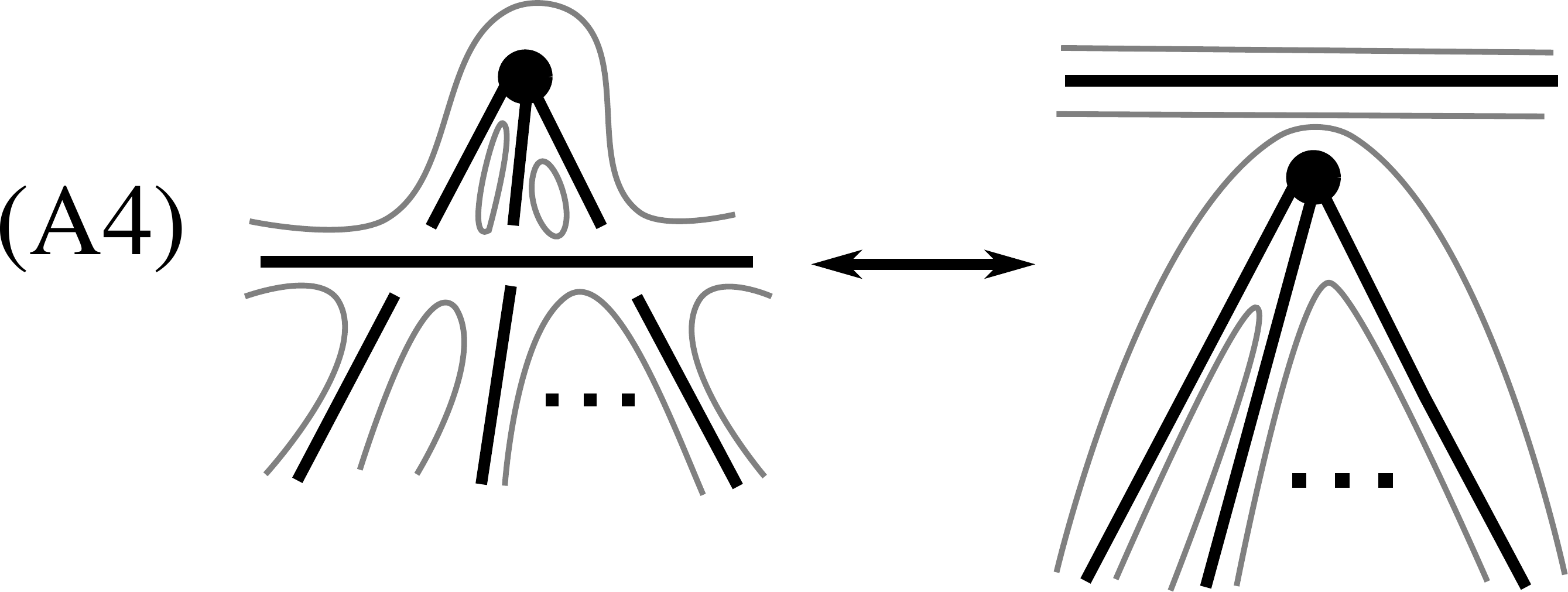}
\caption{An abstract Reidemeister move.}\label{fig:abstractReid}
\end{figure}
\begin{prop}\label{prop:virtualequiv}
There is a bijection between the set of virtual graphoids and the set of abstract graphoid diagrams up to abstract Reidemeister moves.
\end{prop}
\begin{proof}
Take a graphoid $G$. Then, one can build an abstract graph diagram $(F,G)$ following the recipe as discussed. Now, if $G$ and $G'$ differ by a sequence of generalized Reidemeister moves $\Omega_1,\cdots,\Omega_k$, then one can obtain from $(F,G)$ an abstract graph diagram $(F',G')$ via the ribbon versions of $\Omega_1,\cdots,\Omega_k$. Thus, we have a well-defined map from the set of virtual graphoids to the set of abstract graphoid diagrams up to abstract Reidemeister moves.

To see that there is a well-defined inverse for the map constructed in the previous paragraph, we take an abstract graph diagram $(F,G)$. Such a diagram can projected to $S^2$ as follows. First, embed $(F,G)$ in $S^3$ in such a way that the 2-disks corresponding to the classical crossings and vertices of the graphoid lie in $S^2\subset S^3.$ Each instance where two transverse ribbon bands are projected to create a transverse double point in the diagram, we decorate that with a virtual crossing. The proposition follows from noticing that this projection map followed by graphoid Reidemeister moves yield a graphoid diagram that is equivalent to the one coming from performing abstract Reidemeister moves followed by the projection.
\end{proof}

Proposition \ref{prop:virtualequiv} allows us to deduce the following theorem. If two graphoid diagrams $G_1$ and $G_2$ in closed orientable higher
genus surfaces $\Sigma_1$ and $\Sigma_2$ are related by a finite sequence of classical Reidemeister moves in the surfaces, isotopy of the surfaces and addition/removal of handles in the complement of the graphoid diagrams, we say that $(\Sigma_1,G_1)$ is \textit{stably equivalent} to $(\Sigma_2,G_2)$.
\begin{thm}
Virtual graphoid theory is equivalent to the theory of graphoids in thickened surfaces up to stable equivalence.
\end{thm}
\begin{proof}
We first show that there is a bijective map $f$ between the set of abstract graphoids up to abstract equivalence and the set of graphoids in thickened surfaces up to stable equivalence. The theorem then follows because of Proposition \ref{prop:virtualequiv}.

From an abstract graphoid diagram, one can get a graphoid diagram on a closed orientable surface by filling in the boundary of the abstract ribbon surface with disks. Now, we observe that some of the abstract Reidemeister moves alter the genus of the close surface supporting the graphoid diagram. Therefore, abstract Reidemeister moves will correspond to a classical Reidemeister moves on the supporting surface with possible addition and removal of handles in the complement of the diagram. This demonstrates that $f$ is well-defined.

To see that $f$ is surjective, take a graphoid diagram on a closed surface $(\Sigma,G).$ Then, $G$ has a regular neighborhood which is an abstract graphoid diagram. To see that $f$ is injective, suppose that $(\Sigma_1,G_1)$ is stably equivalent to $(\Sigma_2,G_2).$ The Reidemeister moves applied to $G_i$ induces equivalence of abstract graphoid diagrams by the corresponding abstract Reidemeister moves. Finally, since the addition/removal of handles is done away from $G_i$, these actions do not affect the corresponding abstract graphoid diagrams.
\end{proof}
\section{Invariants coming from constituent links/multi-knotoids}\label{section:replacement}

\subsection{Topological invariants induced by local replacements}
A \textit{constituent link} (resp. \textit{constituent multi-knotoid}) of a graphoid $G$, is a link (resp. a multi-knotoid) contained in $G$. In \cite{KGraph}, Kauffman showed that collections of constituent knots and links in a spatial graph is a topological invariant of the spatial graph. Using the same types of arguments, one can deduce that collections of knotoids and knots become a topological invariant of the graphoid.

To be more precise, one performs local replacements near a vertex as shown in Figure \ref{fig:ReidPermute}, where each local replacements connects two of the edges incident to the vertex and leaves the other edges as free ends. Each Reidemeister move permutes the local replacements, and hence the collection of constituents is an invariant of a topological graphoid.
\begin{figure}[ht!]

\centering
\includegraphics[width=7cm]{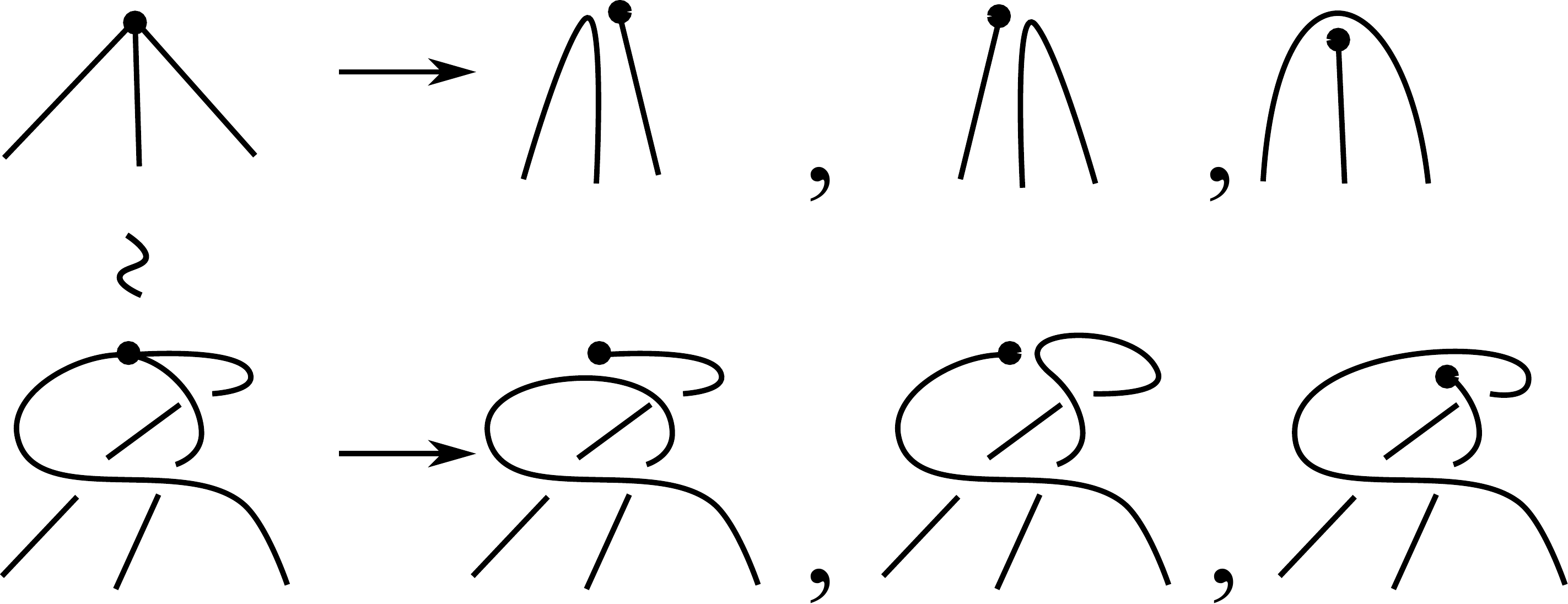}
\caption{Reidemeister move permutes the local replacements.}\label{fig:ReidPermute}
\end{figure}
\begin{figure}[ht!]

\centering
\includegraphics[width=5cm]{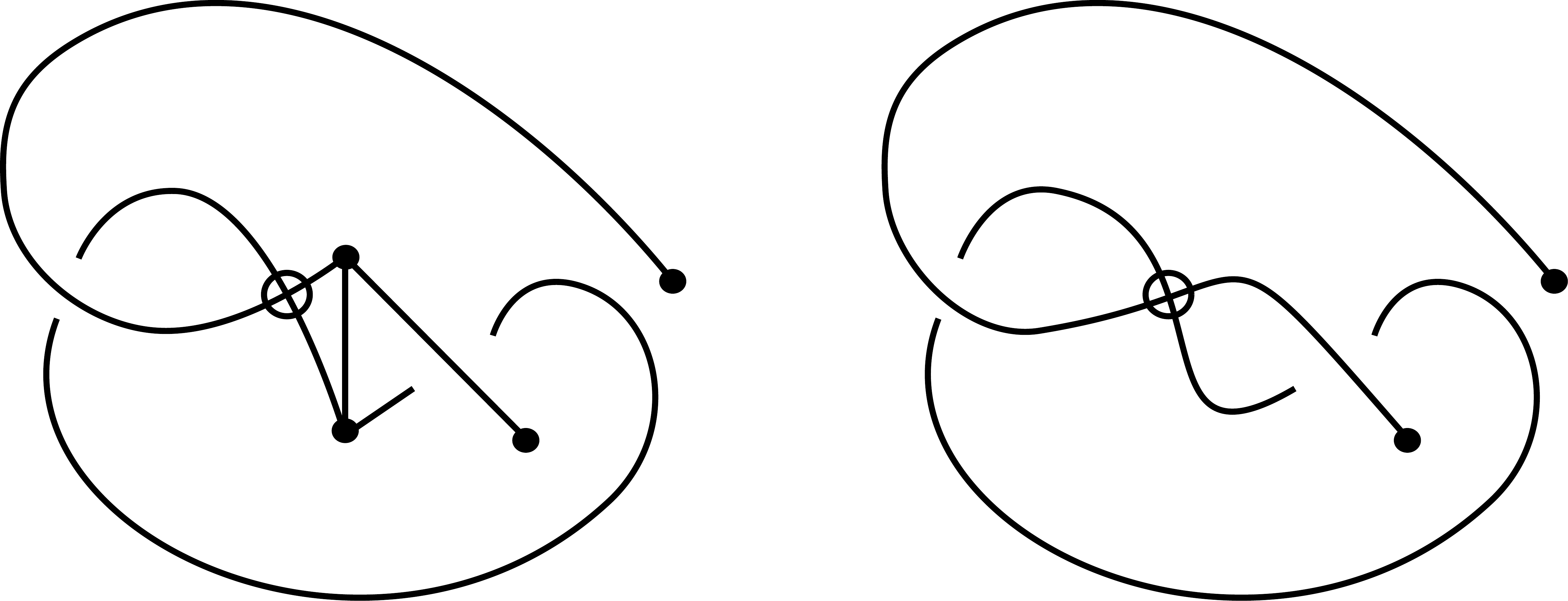}
\caption{The virtual graphoid on the left contains a virtual multi-knotoid whose virtual closure is non-classical.}\label{fig:constituentexample}
\end{figure}
\begin{exmp}
After performing the local replacements discussed in this section, we see that there is a constituent virtual multi-knotoid of height one contained in the virtual graphoid in Figure \ref{fig:constituentexample}.
\end{exmp}
\begin{exmp}
By Theorem \ref{thm:lineisotopy}, a planar classical graphoid $G$ corresponds  to a line isotopy class of a spatial graph whose open ends lie on the line $\lbrace h \rbrace\times \mathbb{R}$ and $\lbrace t \rbrace\times \mathbb{R}$. The spatial graphoid $G$ can be considered to include the two lines and so the resulting object is a long graphoid $\widetilde{G}$.


One can perform local replacements at the vertices of $\widetilde{G}$ to retrieve constituents, as shown in Figure \ref{fig:longtrefoil}. We observe that one of the constituents is a long trefoil which proves that $G$ is non-trivial. This technique was also utilized in \cite{Graphoid}. 
\end{exmp}

\begin{figure}[ht!]

\centering
\includegraphics[width=7cm]{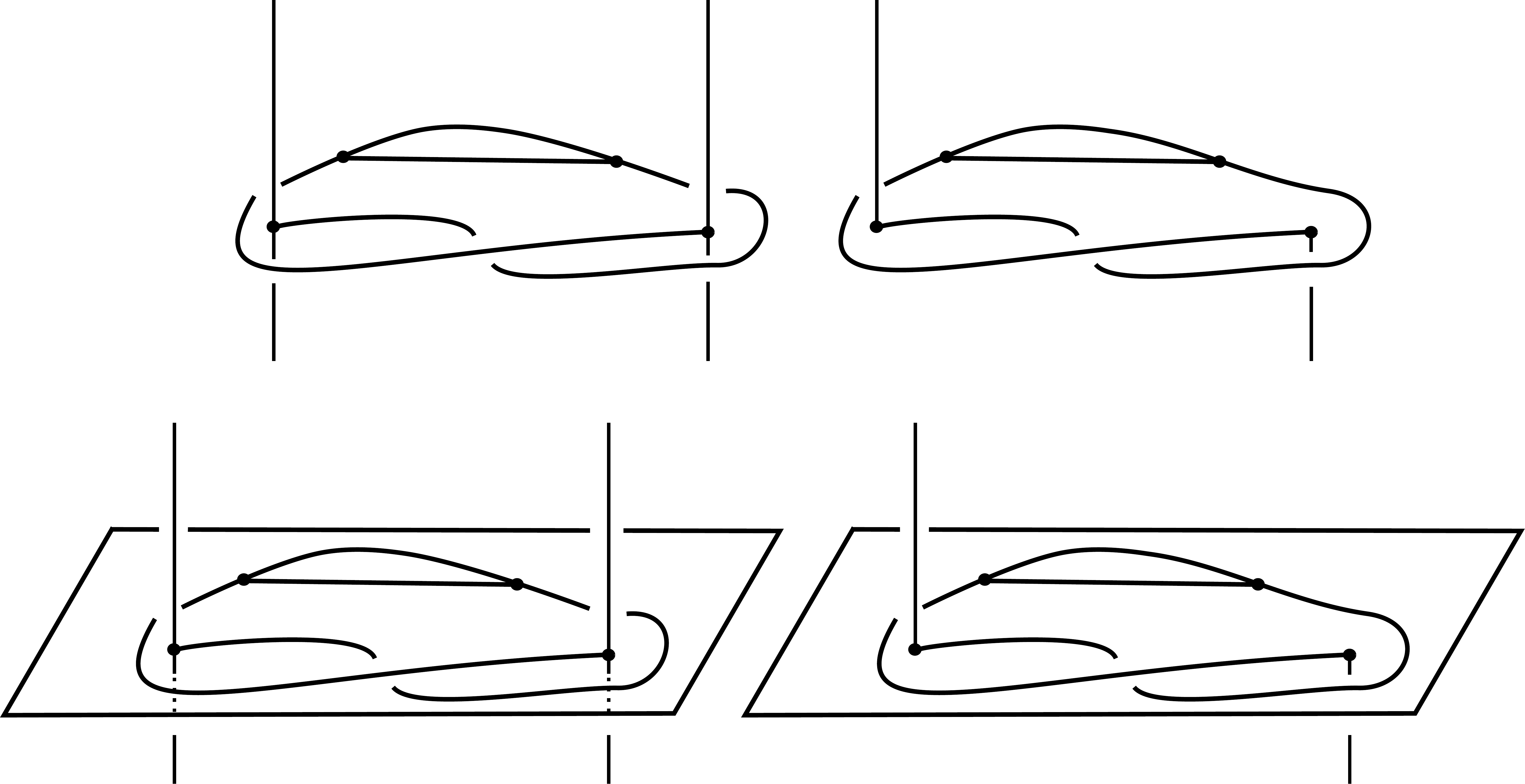}
\caption{A long spatial graph as a constituent knot of the union of a graphoid diagram with the two lines coming from the topological characterization of planar graphoids in terms of line isotopy.}\label{fig:longtrefoil}
\end{figure}

The constituent links also give some information on lower estimates of invariants. We say that two constituents \textit{overlap} if they share a common edge. Define the \textit{maximum constituent crossing number} $mcc(G)$ (resp. \textit{maximum constituent height} $mch(G)$) of the set $\kappa = \lbrace K_1,...,K_n \rbrace$ of constituents where no two constituents in the set overlap to be the sum of the crossing numbers (resp. sum of the heights) of the elements in $\kappa.$
\begin{prop}
Let $c(G)$ denote the crossing number of $G$ and $h(G)$ denote the height of $G$. Then, $mcc(G)\leq c(G)$ and $mch(G)\leq h(G)$.
\end{prop}
For example, the graphoid in Figure \ref{yamadaexample} has height 1 and crossing number 1 since it contains a multi-knotoid whose closure is the virtual Hopf link.

\subsection{Rigid vertex invariants from local replacements}
This idea was also originated in \cite{KGraph}. Suppose that we have a $2k$-valence vertex $v$ in our virtual graphoid diagram $D$ representing a virtual graphoid $G$. We can replace $v$ with some $k$-string tangle $T$ to obtain a new diagram $D_T$. If $T$ is chosen carefully, then the equivalence class of $D_T$ is an invariant of $G$. By a careful selection of $T,$ we mean a tangle that is invariant under the rigid vertex Reidemeister moves. For instance, the "plat closure" or the "braid closure" are some good choices (see Figure \ref{fig:insert}).
\begin{figure}[ht!]
\centering
\includegraphics[width=5cm]{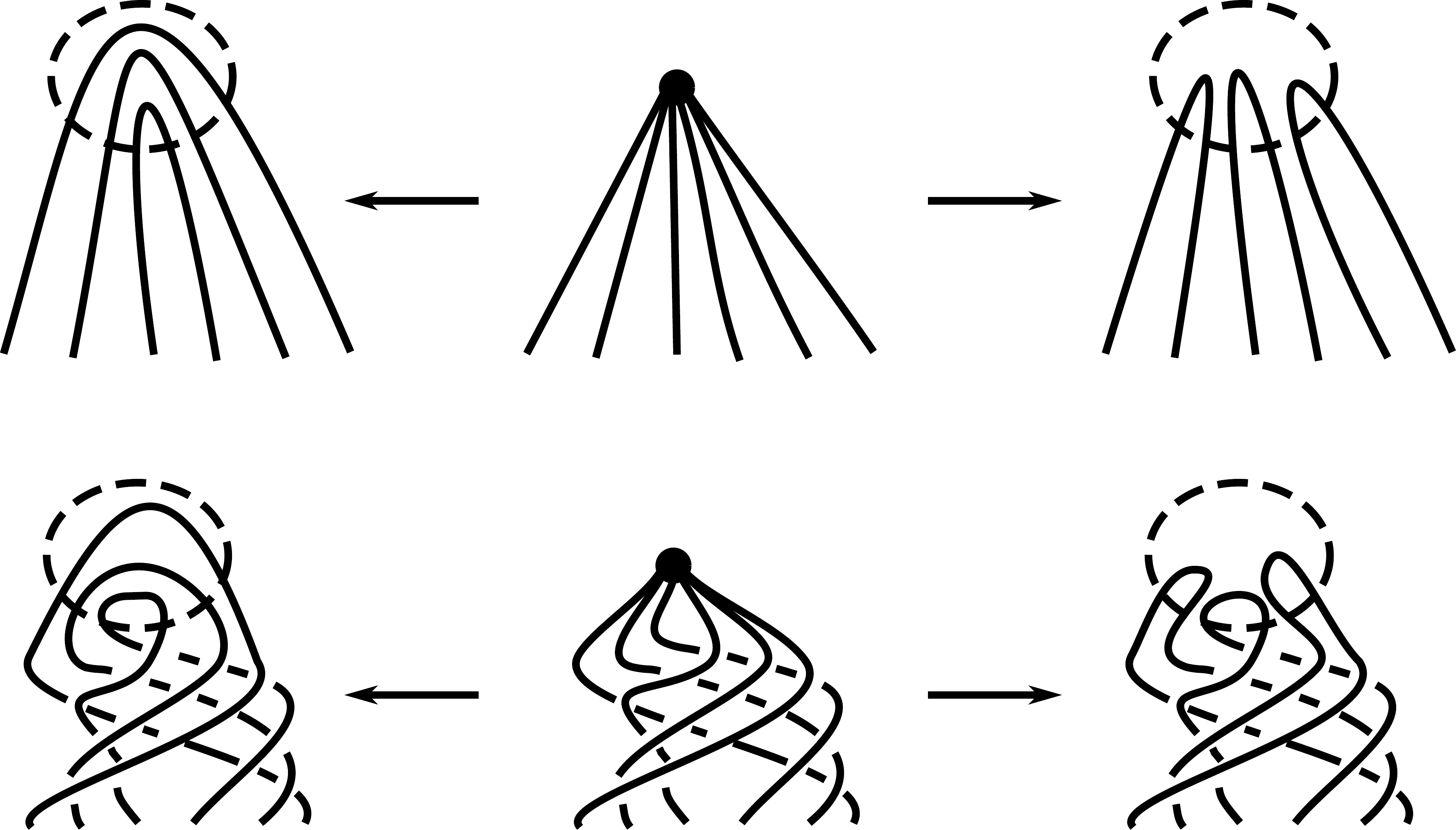}
\caption{The arrow pointing right shows the situation where a vertex is replaced by a plat closure. The arrow pointing left shows the situation where a vertex is replaced by a braid closure.}\label{fig:insert}
\end{figure}

\begin{cor}
    The graphoid $G_n$ on the left of Figure \ref{fig:rigidexmp} has virtual crossing number equal to $n-1.$
\end{cor}
\begin{proof}
    Kauffman showed by extended bracket polynomial in \cite{KExtended} that the virtual crossing number of the virtual closure of the graphoid on the right of Figure \ref{fig:rigidexmp} is $n-1$.
\end{proof}
\begin{figure}[ht!]
\centering
\includegraphics[width=5cm]{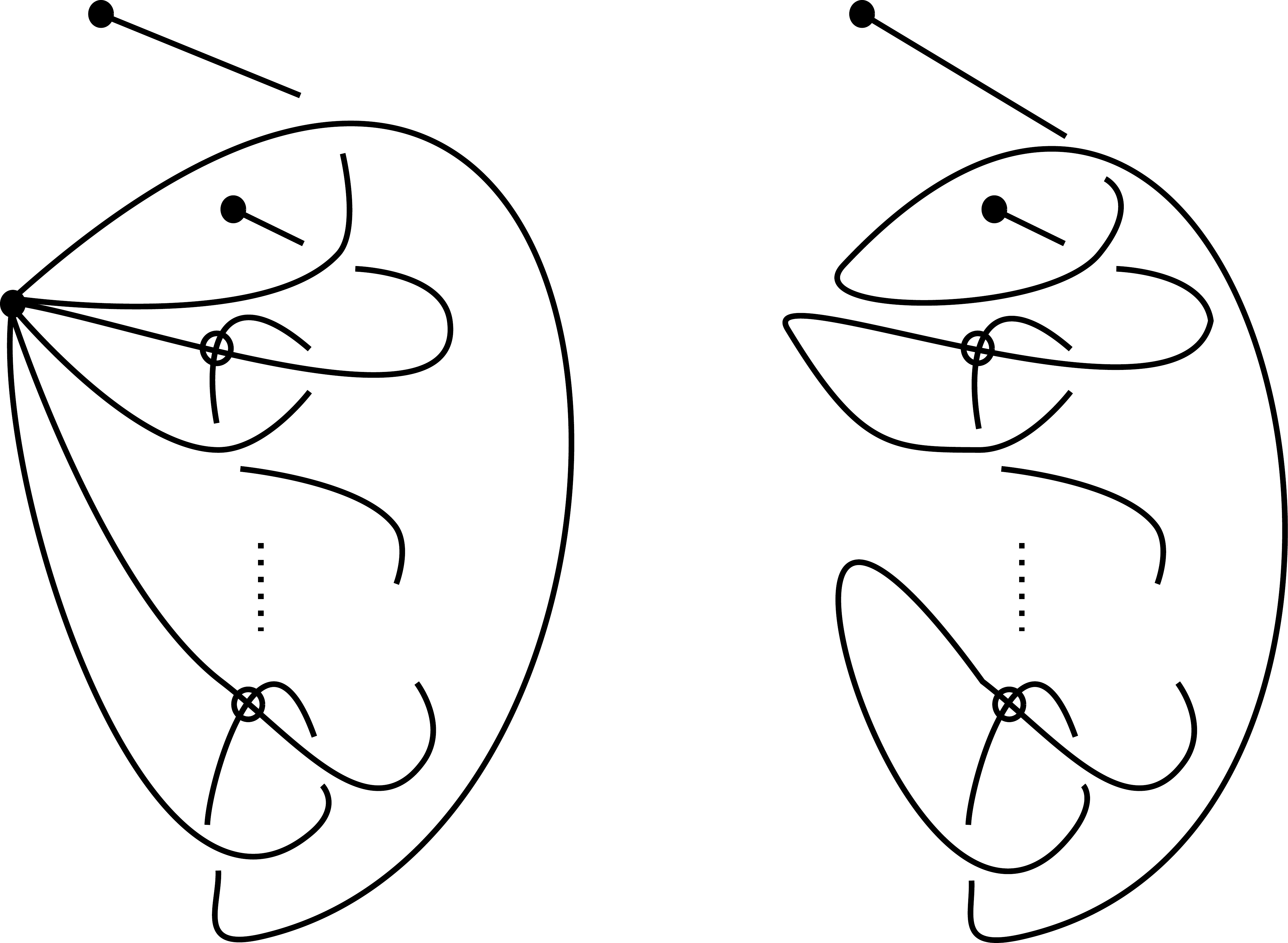}
\caption{Left: The virtual graphoid $G_n$. Right: A vertex is replaced by a plat closure tangle.}\label{fig:rigidexmp}
\end{figure}
\section{Polynomial invariants}\label{section:Yamada}

Arguably, one of the most utilized polynomials to study classical spatial graphs is the Yamada polynomial. There are three polynomials for virtual spatial graphs related to the Yamada polynomial that one can turn into graphoid invariants \cite{Deng,Flem,MM}. We will focus mainly on Fleming-Mellor's formulation $R(G)$ and Deng-Jin-Kauffman's formulation $R(G;A,x)$. 

\subsection{Yamada polynomials $R(G)$ for graphoids}

We begin by defining the Yamada polynomial $R(G)$, which can be computed on a graphoid diagram. This polynomial is particularly useful when applied to virtual spatial graphoids whose constituent knots and knotoids are all trivial. The polynomial $R(G)$ of a virtual graphoid $G$ will be equivalent to the Fleming-Mellor's Yamada polynomial of the virtual closure of $G$ \cite{Flem}. We remark that the Yamada polynomial is invariant under move (VI) in Figure \ref{graphReid} only when the valence of each vertex is at most three. Thus, for graphs with a vertex whose degree is at least four, the polynomial is a rigid vertex invariant.

\begin{defn}
    Let $D$ be a virtual graphoid diagram. The Yamada polynomial $R(D)$ is defined recursively with the relations given in Figure \ref{yamadaskein}. 
\end{defn}

\begin{figure}[ht!]

\centering
\includegraphics[width=11cm]{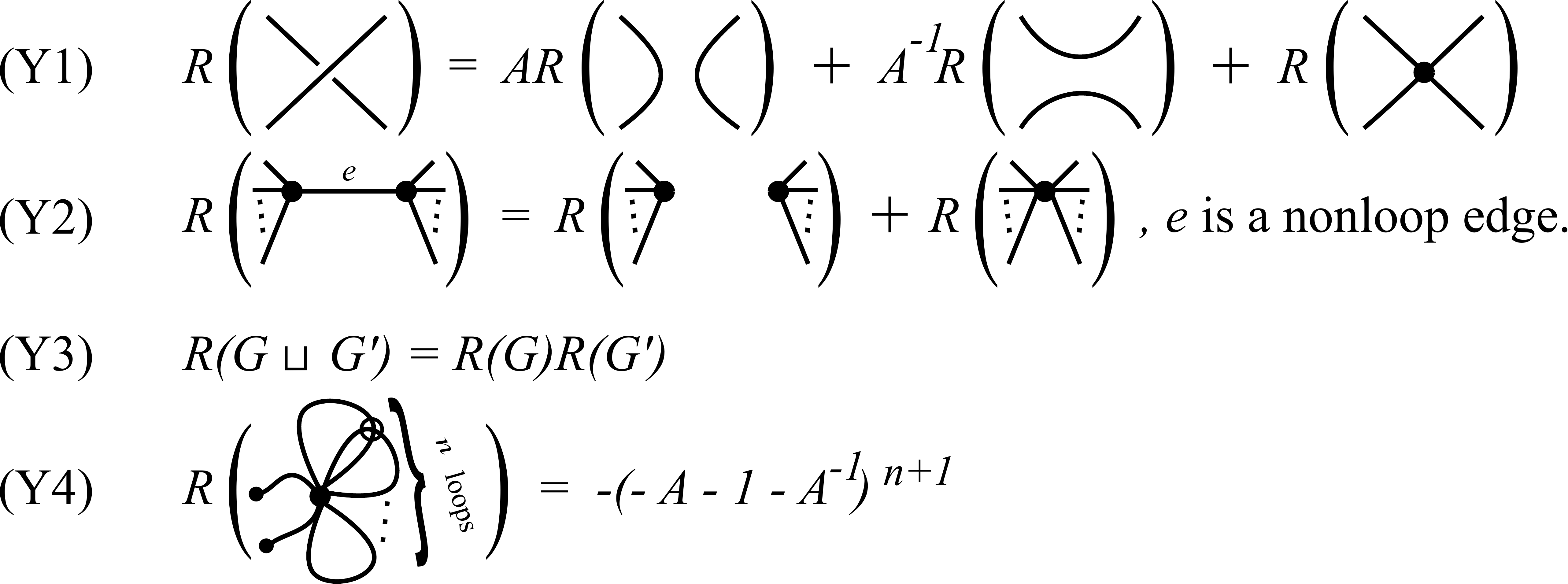}
\caption{Relations satisfied by the Yamada polynomial.}\label{yamadaskein}
\end{figure}

Note that we do not allow the contraction-deletion operations on an edge adjacent to the head/tail of $D$. Thus the relation (Y2) is only relevant for an edge that is adjacent to vertices of degree $k \geq 2$. The relation (Y4) evaluates $n$- bouquet graphoid diagram $\dot B_{n}$ for $n \geq 0$, that is a graphoid diagram with $n$ loops at a unique vertex on it. In particular, $R(K)= A+1+A^{-1}$ where $K= \dot B_0$ is the trivial knotoid diagram. We also assume $R(\emptyset)=1$, and $R(D)= -1$ if $D$ is a single vertex graph.

Figure \ref{yamadaexample} goes through how to compute $R(G)$ for a specific graphoid.
\begin{figure}[ht!]

\centering
\includegraphics[width=11cm]{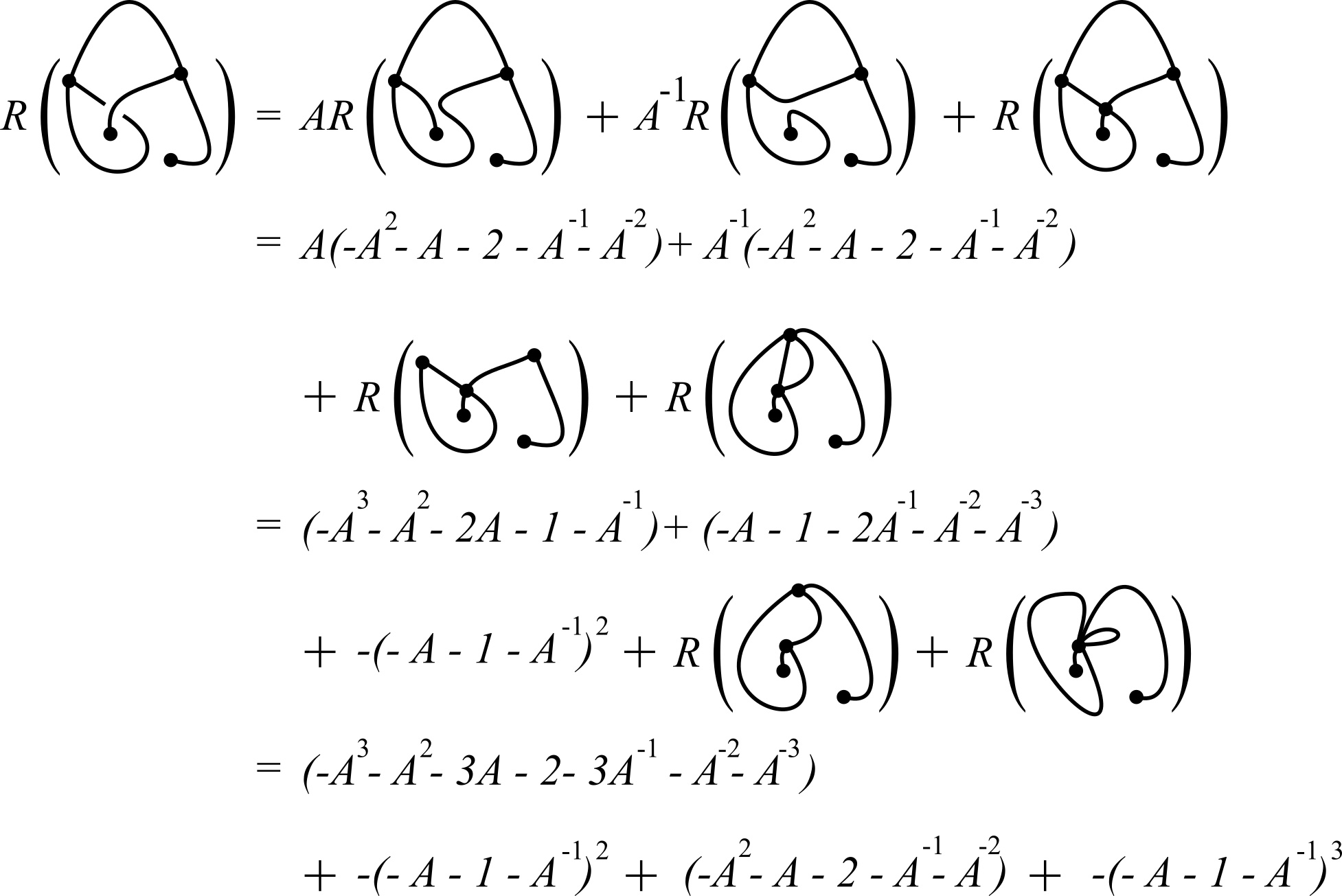}
\caption{A computation of the Yamada polynomial.}\label{yamadaexample}
\end{figure}

Next, we focus on the state-space expansion of $R(G)$, which involves three types of resolutions at a crossing. These resolutions are shown in Figure \ref{yamadasmooth} and we refer to them from left to right as the B-resolution, the A-resolution, and the X-resolution. After all crossings are spliced according to these three resolutions, we perform a virtual closure to join the leg and the tail of the graphoid. Let $\mathcal{S}$ be the set of states for a graphoid $G$. For a state $S$, we recall the special case of the flow polynomial
\begin{center}
    $H(S) = \displaystyle \sum_{F\subset E(S)}(-1)^{\beta_0(S-F)}(-A-2-A^{-1})^{\beta_1(S-F)}$.
\end{center}
\begin{figure}[ht!]

\centering
\includegraphics[width=5cm]{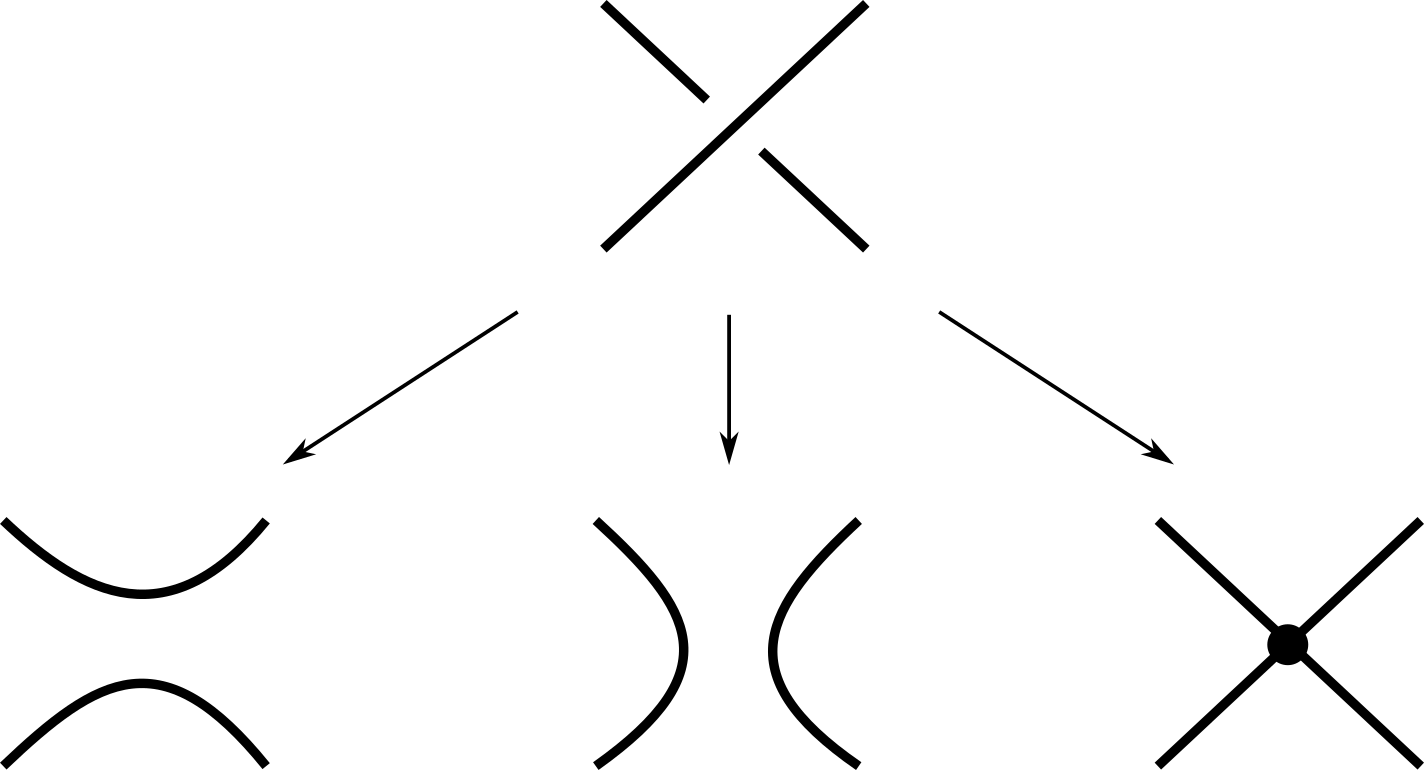}
\caption{The three types of resolutions involved in the definition of the Yamada polynomial.}\label{yamadasmooth}
\end{figure}
Then, the Yamada polynomial is $R_G(A) = \displaystyle \sum_{S \in \mathcal{S}}A^{a(S)-b(S)}H(S)$, where $a(S)$ (resp. $b(S)$) stands for the number of A-resolutions (resp. B-resolutions) in $S$.
The next lemma relates the maximum and minimum degree of $R(G)$ to the crossing number of $G$, the all A states $S_A$ and the all B states $S_B$. We perform our analysis on the virtual spatial graph that is the closure of the graphoid.

\begin{lem}\label{lem:degreeinequality}
Let $G$ be a virtual spatial graph. Let $D$ be a diagram of $G$. Then, the following statements hold:
(i) max deg($R) \leq c(D) +\beta_1(S_A(G))$ \\ (ii) min deg($R) \leq -c(D) -\beta_1(S_B(G))$
\end{lem}
\begin{proof}
We will only provide the argument for (i). With very minor changes, the readers will see that the provided argument works for (ii).  Observe that for the all A states, $c(D)=a(S_A)-b(S_A)$. Suppose that $S_a,S_b,$ and $S_x$ are three states where all resolutions are identical except at one place. At that one place, say $S_a, S_b,$ and $S_x$ has the A-resolution, B-resolution, and the X-resolution, respectively.

Observe that $S_x$ has the same Euler characteristic with the graph obtained from $S_a$ by adding one edge. If such an edge creates one more cycle, then $\beta_1(S_a) = \beta_1(S_x)-1$. If such an edge merely connects two disconnected components, then $\beta_1(S_a) = \beta_1(S_x)$. In conclusion, $\beta_1(S_x)-1\leq \beta_1(S_a) \leq \beta_1(S_x)$. Similarly, $\beta_1(S_x)-1\leq \beta_1(S_b) \leq \beta_1(S_x)$. Combining these inequalities together, we have that $\beta_1(S_b)-1\leq \beta_1(S_a)\leq \beta_1(S_b)+1$. Obviously, $a(S_b) = a(S_x) = a(S_a) -1$ and $b(S_a) = b(S_x) = b(S_b) -1$. From these relationships, we can see that $a(S_x)-b(S_x)+\beta_1(S_x)\leq a(S_a)-b(S_a)+\beta_1(S_a)$ and $a(S_b)-b(S_b)+\beta_1(S_b)\leq a(S_a)-b(S_a)+\beta_1(S_a)$.

Now, max deg $H(S) \leq \beta_1(S)$ because the exponents of $A$ in $H(S)$ are in terms of the first Betti number of $G$ with some edges removed. In conclusion, for any state $S$,

\begin{align*}
    max deg(R) &\leq max deg A^{a(S)-b(S)}H(S) \\ &\leq a(S)-b(S)+max H(S)\\ &\leq a(S_A)-b(S_A)+\beta_1(S_A(G)).
\end{align*}
The final inequality follows from the fact that any state $S$ is connected to $S_A$ by a sequence of states, where adjacent states in the sequence differ by turning a B-resolution or an X-resolution to an A-resolution.
\end{proof}

Therefore, we can conclude that span($R$) = max deg$(R)$ $-$ min deg$(R) \leq 2c(D) + \beta_1(S_A(G))+ \beta_1(S_B(G))$. A natural question arises: When can we say the equality holds? We will see that if $G$ belongs to the family of virtual graphoids called \textit{adequate} graphoids, then span($R(G)$) = max deg$(R(G))$ $-$ min deg$(R(G)) = 2c(D) + \beta_1(S_A(G))+ \beta_1(S_B(G))$. 

 The notion of adequacy dates back to Lickorish and Thistlethwaite \cite{LK} and have been generalized to other knotted objects \cite{Bae, BKS}. We now define the notion of adequacy. While the definition is reminiscent of adequate virtual knots, it is more complicated and the readers are encouraged to refer back to Figure \ref{fig:adequate} as they read the definition.

\begin{figure}[ht!]
\labellist \small\hair 2pt
\pinlabel $D$ at 7 1146
\pinlabel $S_A'$ at 707 1146
\pinlabel $S_B'$ at 2007 1146
\pinlabel $F_A$ at 707 746
\pinlabel $F_B$ at 2007 746
\pinlabel $J_A$ at 707 246
\pinlabel $J_B$ at 2007 246

\endlabellist
\centering
\includegraphics[width=11cm]{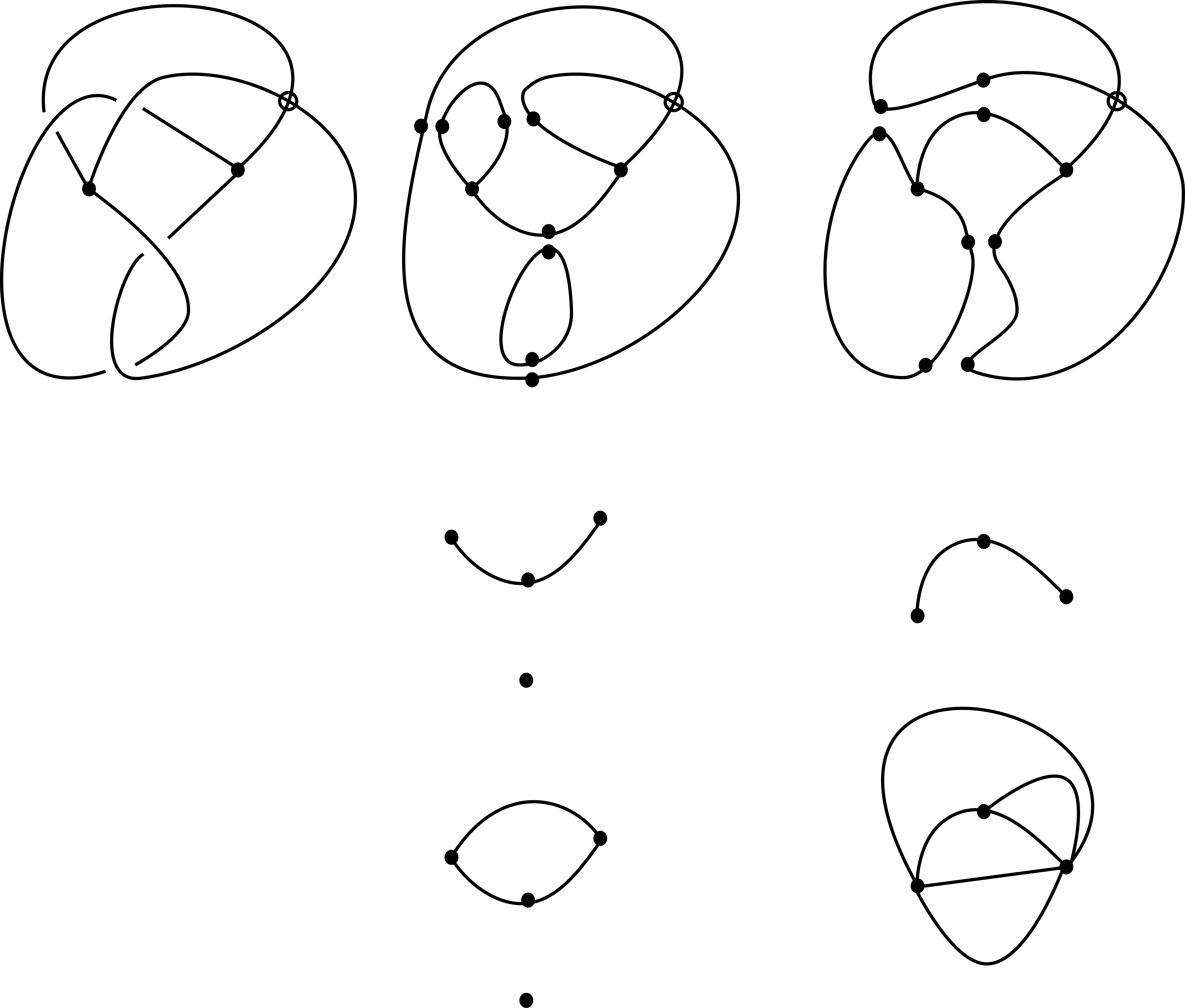}
\caption{Various graphs appearing in the definition of adequate virtual spatial graph.}\label{fig:adequate}
\end{figure}
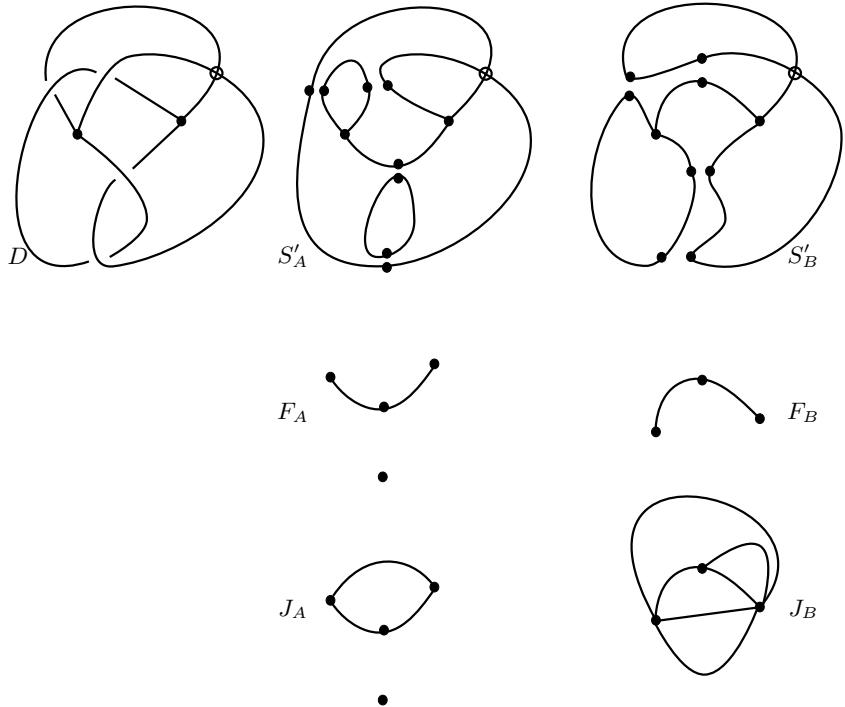

Let $S_A'$ (resp. $S_B'$) denote the result of adding two vertices, where each vertex lies on each arc in the local picture of the $A$-resolution (resp. the $B$-resolution). Let $F_A$ (resp. $F_B$) be a graph where

\begin{enumerate}
    \item Each vertex corresponds to a connected component of $S_A'-CE(S_A')$ (resp. $S_B'-CE(S_B')$, where $CE(-)$ denotes the set of cut edges.
    \item The edges are the cut edges with a natural incidence induced from $S_A'$ (resp. $S_B'$).
    \end{enumerate}
    
Let $x$ be a crossing point of $G,$ and let $S_{A,x}$ be the state obtained from $S_A$ by replacing the A-resolution at $x$ to the X-resolution. A crossing $x$ is called $A$-essential (resp. $B$-essential) if doing this replacement increases the first Betti number of $S_A$ (resp. $S_B$) by 1. Let $J_A$ (resp. $J_B$) be a graph obtained from $F_A$ (resp. $F_B$) by adding edges corresponding to the $A$-essential (resp. $B$-essential) crossings. Let $a_k$ (resp. $b_k$) be the number of subgraphs $K_A$ (resp. $K_B$) of $J_A$ (resp. $J_B$) such that
\begin{enumerate}
    \item $F_A\subset K_A$ (resp. $F_B\subset K_B$).
    \item $K_A$ (resp. $K_B$) has no cut edges.
    \item $\beta_1(K_A)=k$ (resp. $\beta_1(K_B)=k$).
 \end{enumerate}

 \begin{defn}
 A diagram $D$ of a virtual spatial graph is \textbf{$A$-adequate} (resp. \textbf{$B$-adequate}) if $\sum_{k=0}^{\infty}a_k\neq 0$ (resp. $\sum_{k=0}^{\infty}b_k\neq 0$). A virtual spatial graph diagram is \textbf{adequate} if it is both $A$-adequate and $B$-adequate. A virtual spatial graph is \textbf{adequate} if it admits an adequate diagram.
 \end{defn}
 \begin{exmp}
     Let $D$ be the diagram in Figure \ref{fig:adequate}. We claim that $D$ is adequate. Observe that there is only one possibility for $K_A$, which is the graph $J_A$ itself. We remark that $F_A$ does not satisfy condition (2) in the definition of $K_A$ because it has a cut edge. It follows that $a_1=1$ and the quantity $\sum_{k=0}^{\infty}a_k=-1$ is nonzero. Therefore, $D$ is $A$-adequate.

     Next, note that $b_0=0$. Also, the value of $b_k$ for $k\neq 0$ is $\binom{4}{k}$ because to form $J_B$ from $F_B$ there are 4 new edges added. Each choice of $k$ edges we choose to remain contributes one to $b_k.$ Thus, $\sum_{k=0}^{\infty}b_k = -4+6-4+1=-1\neq 0$ and $D$ is $B$-adequate.
 \end{exmp}
 \begin{exmp}
Adequate virtual links are examples of adequate virtual spatial graphs. In this case, $S_A$ and $S_B$ do not contain cut edges. This means that $F_A$ and $F_B$ are isolated vertices. Thus, $J_A$ and $J_B$ may contain some number of loops. Any such loop gives rise to a self-abutting cycle, which cannot exist for adequate link diagrams.
\end{exmp}
\begin{exmp}
    Starting with an adequate classical spatial graph, and turning a random crossing to a virtual crossing does not always give an adequate virtual spatial graph. However, if one is selective and turns a crossing that is away from cut edges of the all A-states and all B-states, one will get an adequate virtual spatial graph.
\end{exmp}
\begin{thm}
If $D$ is adequate, then max deg$(R)$ $-$ min deg$(R) = 2c(D) + \beta_1(S_A(G))+ \beta_1(S_B(G))$
\end{thm}
\begin{proof}
The proof will be similar to the proof found in \cite{Motohashi}. We will show that the coefficient of $A^{\max \ deg(R)}=A^{c(D)+\beta_1(S_A(G))}$ is nonzero. In fact, that coefficient turns out to be $(-1)^{\chi(G)}\sum_{k=0}^{\infty}a_k\neq 0$ which is nonzero according to how we defined adequacy. Essentially the same argument shows that the coefficient of $A^{\min \ deg(R)}=A^{-c(D)-\beta_1(S_B(G))}$ is nonzero. As we have seen the proof of Lemma \ref{lem:degreeinequality}, and due to the fact that $H(S)=0$ for a state $S$ with cut edges, it's enough to consider states that may contribute to the extreme terms $S$ with the following properties: (1) $S$ has no B-smoothings, (2) $S$ has no cut edges, and (3) each X-smoothing gives an A-essential crossing. To eleborate more on (3), if an X-smoothing does not increase $\beta_1$, then $a(S_x)-b(S_x)+\beta_1(S_x) = a(s_a)-b(s_a)+\beta_1(S_1)-1$, which is strictly less than the exponent associated to the extreme term.

Since each X-smoothing increases the first Betti number by one, $\beta_1(S) = \beta_1(S_A)+x(S)$, where $x(S)$ is the number of X-resolutions. This implies that $a(S)-b(S)+\beta_1(S)=a(S)+\beta_1(S_A)+x(S)=c(G)+\beta_1(S_A) = max(A^{a(S)-b(S)}H(S)).$ Here, the final equalty follows from the fact that max$(H(S))=\beta_1(S)$ and the coefficient of $A^{a(S)-b(S)+\beta_1(S)}=(-1)^{\chi(S)}$. 

By design, the $a_k$ subgraphs $K_A(S)$ appearing in the definition of adequate virtual spatial graphs correspond to the $a_k$ states of $D$ with no B-smoothings, no cut edge, and $k$ X-smoothings at essential crossings.
Therefore,
\begin{align*}
    A_{c(G)+\beta_1(S_A)}(R(G)) &=\sum^{\infty}_{k=0}(-1)^{\chi(G)+k}a_k.
\end{align*}.

\end{proof}

We now proceed to bound $\beta_1(S_A(G))+ \beta_1(S_B(G))$ in terms of other quantities that are easier to work with.

We defined a dual graph $G^*$ from a cellularly embedded projection $\widehat{G}$ on a closed surface $F$ as follows:
\begin{enumerate}
    \item The vertices of $G^*$ correspond to the disk regions of $F\backslash \widehat{G}$.
    \item For each crossing point $x$ of $\widehat{G}$, there are two edges of $G^*$ such that each of which joins the two disk regions incident at $x$. 
\end{enumerate}

We let $s(\widehat{G})$ be $\beta_0(G^*)$. By $s(G)$, we mean the maximum of $s(\widehat{G})$ over all shadows $\widehat{G}$ of $G$. In the proof below, $\widehat{G}$ is assumed to be connected. The proof can be generalized to handle multiple connected components and $\beta_0(\widehat{G})$ will show up in the inequality.
\begin{lem}
Let $D$ be a diagram of $G$ on a surface $F$ that is cellularly embedded. Then, $\beta_1(S_A(D))+\beta_1(S_B(D)) \leq  s(\widehat{G})-\chi(G)+c(D)-\chi(F)+2.$\label{lem:interpretingfirstbetti}
\end{lem}
\begin{proof}
Since $\widehat{G}$ has the same Euler characteristic as the graph obtained from $G$ by adding $c(\widehat{G})$ edges, it follows that $\chi(\widehat{G}) = \chi(G)-c(\widehat{G})$. We then have that $\beta_1(\widehat{G}) = \beta_0(\widehat{G})-\chi(\widehat{G})=1-\chi(G)+c(\widehat{G}).$ 

Now, since $D$ is cellularly embedded, $|V(G^*)|=\chi(F)-1+\beta_1(\widehat{G}).$ It follows that $|V(G^*)|=1-\chi(G)+c(\widehat{G})+\chi(F)-1=-\chi(G)+c(\widehat{G})+\chi(F).$ Let $G_A^*$ and $G_B^*$ be the subgraphs of $G^*$ that (1) share the same vertices with $G^*$ (2) $E(G_A^*)\cup E(G_B^*)=E(G^*)$ (3) $E(G_A^*)\cap E(G_B^*)=\emptyset$ and (4) $G_B^*\cap S_B=G_A^*\cap S_A=\emptyset$. By the Mayer-Vietoris sequence $H_0(V(G^*))\rightarrow H_0(G_A^*) \oplus H_0(G_B^*) \rightarrow H_0(G^*)\rightarrow 0.$ Since the sequence is exact, $\beta_0(G_A^*)+\beta_0(G_B^*)\leq \beta_0(G^*)+|V(G^*)|.$

Now, $\beta_1(S_A(D)) = \beta_0(F\backslash S_A(D))-\chi(F)+1.$ Similarly, $\beta_1(S_B(D)) = \beta_0(F\backslash S_B(D))-\chi(F)+1.$ But $F\backslash S_A(D)$ deformation retracts to $G_A^*$ (similarly for $S_B(D)).$ Therefore, $\beta_1(S_A(D)) =\beta_0(G_A^*)-\chi(F)+1$ and $\beta_1(S_B(D)) =\beta_0(G_B^*)-\chi(F)+1$.

In conclusion, $\beta_1(S_A(D))+\beta_1(S_B(D)) = \beta_0(G_A^*)-\chi(F)+1+\beta_0(G_B^*)-\chi(F)+1\leq \beta_0(G^*)+|V(G^*)|-2\chi(F)+2=\beta_0(G^*)+\chi(F)-\chi(G)+c(\widehat{G})-2\chi(F)+2 = s(\widehat{G})-\chi(G)+c(\widehat{G})-\chi(F)+2$.
\end{proof}

\begin{lem}
$s(\widehat{G}) \leq \beta_1(G) +\chi(F)-1$.
\end{lem}

\begin{proof}
We will induct on $c(\widehat{G}).$ If $c(\widehat{G}) = 0$, then $s(\widehat{G}) = \beta_0(G^*) = \beta_1(G) +\chi(F)-1.$

Suppose now that the claim is true for $c(\widehat{G})=k.$ Let $\widehat{G}$ be a regular projection of $G$ with $c(\widehat{G})=k+1.$ Choose a crossing $x$ for $\widehat{G}$ and call the graph $G_A$ and $G_B$ to be the result of performing the $A$-smoothing and $B$-smoothing at $x$. If $\beta_0(G_A)>\beta(G)$ and $\beta_0(G_B)>\beta(G)$, then $\beta_1(G)=\beta_1(G_A)=\beta_1(G_B),$ which is a contradiction.

This means that without loss of generality, $\beta_1(G_A)\leq \beta_1(G)$. This implies that $\beta_0(\widehat{G}^*)\leq \beta_0(\widehat{G_A}^*)$. By induction, $s(\widehat{G}) = \beta_0(\widehat{G}^*)\leq \beta_0(\widehat{G_A}^*)\leq \beta_1(G)+\chi(F)-1.$
\end{proof}
These preliminary lemmas imply the following theorem.

\begin{thm}\label{thm:crossingbound}
    $c(G)\geq \frac{1}{3}(span(R(G))-2\beta_1(G))$.
\end{thm}
\begin{proof}
    By Lemma \ref{lem:degreeinequality}, and Lemma \ref{lem:interpretingfirstbetti}, span$(R(G)) \leq 2c(G)+\beta_1(S_A(D))+\beta_1(S_B(D)) \leq 2c(G)+s(G)-\chi(G)+c(G)-\chi(F)+2\leq 3c(G)+\beta_1(G)+\chi(F)-\chi(F)+2-1-\chi(G)=3c(G)+2\beta_1(G)$.
\end{proof}
We have now arrived at a useful result regarding adequate virtual graphoids.

\begin{cor}
    If $G$ is a virtual graphoid whose virtual closure is adequate, then $G$ is nontrivial.
\end{cor}
\begin{proof}
    If $G$ is adequate, then $\beta_1(\widehat{G})>\beta_1(G)$. Furthermore, since $\widehat{G}$ is obtained from $S_A(D)$ by adding $c(D)$ edges, their Betti numbers are related by $\beta_1(S_A(G))+c(G)\geq \beta_1(\widehat{G})$. Similarly, $\beta_1(S_B(G))+c(G)\geq \beta_1(\widehat{G})$. By definition of adequacy, $deg(R(G))=2c(G)+\beta_1(S_A(G))+\beta_1(S_B(G)) > 2\beta_1(G).$ If $G$ admits a diagram with no classical crossings, then by Theorem \ref{thm:crossingbound}, $deg(R(G))\leq 2\beta_1(G),$ which is a contradiction.
\end{proof}

It is of interest to determine families of graphoids, whose crossing number is realized on a certain type of diagrams. Let $o(G)$ be the number vertices with odd degree. In the theorem below, we assume $\widehat{G}$ is connected. It can be easily generalized to reflect the situation where $\beta_0(\widehat{G})>1.$
\begin{thm}

    Let $D$ be a cellularly embedded graph diagram of $G$ on a closed surface $\Sigma$. Suppose that there is a diagram $D'$ of an Eulerian graph $G'$ obtained by adding $o(G)/2$ edges without introducing new crossing points or vertices such that the over/under crossing decorations of $D'$ come from a checkerboard coloring of $D'$. Then, $\beta_1(S_A(D))+\beta_1(S_B(D)) \geq c(D)-\chi(G)-\frac{o(G)}{2}+\chi(\Sigma)$. \label{thm:checkerboard}
\end{thm}
\begin{proof}
We have that $\beta_1(S_A(D'))$ is at least the number of white regions of the checkerboard coloring. Similarly, $\beta_1(S_B(D'))$ is at least the number of black regions of the checkerboard coloring. From Euler characteristic calculations, the total number of regions is $c(\widehat{G'})+\chi(\Sigma)-\chi(G').$ It follows that $\beta_1(S_A(D))+\beta_1(S_B(D))\geq\beta_1(S_A(D'))+\beta_1(S_B(D'))-o(G)\geq c(D)-\chi(G)+\frac{o(G)}{2}+\chi(\Sigma)-o(G)$. The final inequality holds because $c(D)=c(D')$ and $\chi(G')=\chi(G)-\frac{o(G)}{2}$ by construction.     
\end{proof}
As a consequence, we get a condition on a type of diagrams that realizes the crossing number.
\begin{cor}
    Suppose that $D$ is an adequate virtual spatial graph diagram. Suppose also that $D$ satisfies the property stated in Theorem \ref{thm:checkerboard}. Then, $c(D)-c(G)\leq \frac{1}{3}(s(G)+\frac{o(G)}{2}-2\chi(\Sigma)+2).$
\end{cor}
\begin{proof}
    Since $D$ is adequate, 
    \begin{align*}
         spanR(D) &= 2c(D) +\beta_1(S_A)+\beta_1(S_B) \\ &\geq 3c(D)-\chi(G)-\frac{o(G)}{2}+\chi(\Sigma)
    \end{align*}
    
  . The second inequality comes from the estimate in Theorem \ref{thm:checkerboard}. On the other hand, there is an upper estimate for span$R(G)$ from Lemma \ref{lem:interpretingfirstbetti}. Namely, 
  \begin{align*}
         spanR(D) &\leq 3c(G) -\chi(G)+s(G)-\chi(\Sigma)+2
    \end{align*}
    Moving $c(G)$ and $c(D)$ to the same side, we get that $c(D)-c(G)\leq \frac{1}{3}(s(G)+\frac{o(G)}{2}-2\chi(\Sigma)+2)$.
\end{proof}
\section{Constructing non-classical graphoids.}\label{section:pure}

Deng, Kauffman, and Jin developed a Yamada-type polynomial denoted by $R(G;A,1)$ for virtual spatial graphs \cite{Deng}. In this section, we will use $R(G;A,1)$ to show that certain virtual graphoids cannot be presented as a diagram without virtual crossings. We accomplish this by showing that $R(G) \neq R(G;A,1)$.

\begin{defn}
    Let $D$ be a virtual graphoid diagram. The \textbf{generalized Yamada polynomial} $R(D;A,1)$ satisfies the following relations:
    \begin{enumerate}
        \item The relation ($Y_1$) in Figure \ref{yamadaskein}.
        \item The relation ($Y_2^*$) in Figure \ref{fig:virtualyamskein}. The edge with a slash through it is called a \textbf{marked edge}.
        \item If each edge of $D$ is a marked edge, then $R(D) = (-1)^{bc(D)}(A+A^{-1}+2)^{bc(D)-k(D)}$, where $k(D)=\beta_0(D)$ and $bc(D)$ is the number of boundary components of $D$.
    \end{enumerate}\label{def:DengDef}
\end{defn}
If $D$ is a single vertex, then $R(D;A,1) = -1$. Also, $R(\emptyset)=1.$
\begin{figure}[ht!]

\centering
\includegraphics[width=9cm]{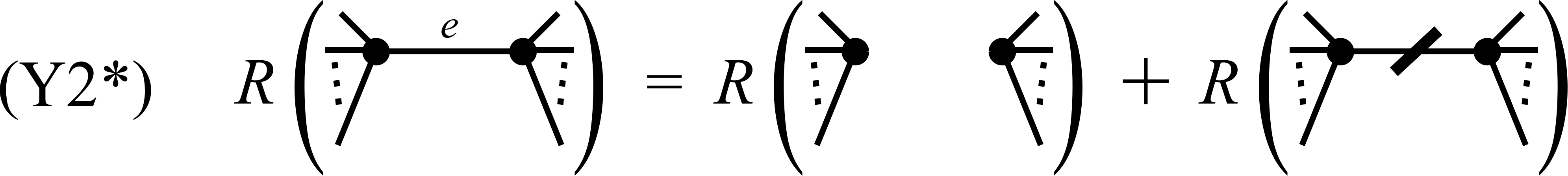}
\caption{A deletion-contraction relation for $R(G;A,1)$.}\label{fig:virtualyamskein}
\end{figure}

Again, we are not allowed to perform the contraction-deletion on the edges adjacent to the tail/head. Next, we gather some useful consequences of the definition of $R(G;A,1)$ that was also presented in \cite{Deng}. 

\begin{enumerate}
    \item $R(K;A,1) =-A-A^{-1}-1$, where $K$ is the trivial knotoid.
    \item If a non-loop edge $e$ does not intersect other edges at classical or virtual crossings, then (Y2) is the same as (Y2*).
\end{enumerate}
An important consequence coming from analyzing the state sum formula that we will use is the following.

\begin{cor}[Corollary 3.17 of \cite{Deng}]
    Let $D$ be a classical spatial graph diagram, then $R(D;A,1)$ coincides with $R(D)$.
\end{cor}

\begin{thm}
    The graphoids in Figure \ref{fig:pure} are pure graphoids, where $T$ is a tangle.
\end{thm}
\begin{proof}
     After the resolutions are performed at each crossing according to the rules in Definition \ref{def:DengDef}, we can group the terms in $R(G;A,1)$ as shown in Figure \ref{fig:pure}. Let $\overline{G}$ denote the graphoid attached to the $\alpha$ term in Figure \ref{fig:pure}. It was shown in \cite{Deng} that $R(\overline{G};A,1)\neq R(U;A,1),$ but $R(\overline{G})=R(U),$ where $U$ can be represented by the diagram of the standard unknotted theta curve in $\mathbb{R}^3$. Due to this inequality, $R(G;A,1)\neq R(G)$ either.
\end{proof}
\begin{figure}[ht!]

\centering
\includegraphics[width=9cm]{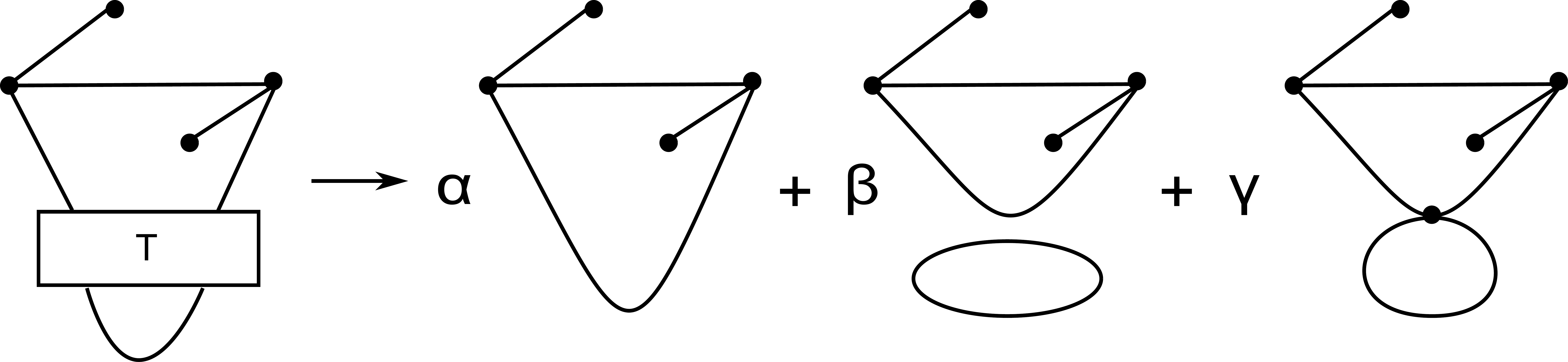}
\caption{A family of pure graphoids.}\label{fig:pure}
\end{figure}


\begin{thebibliography}{99}
\bibitem{Adams}
C. Adams and A. Bonat and M. Chande and J. Chen and M. Jiang and Z. Romrell and D. Santiago and B. Shapiro and D. Woodruff, \textit{Generalizations of Knotoids and Spatial Graphs}, arxiv.2209.01922 math.GT (2022)
\bibitem{Bae}
Y. Bae and H. S. Lee and C. Y. Park, \textit{On the Jones Polynomial of Adequate Virtual Links}, Journal of Knot Theory and Its Ramifications 19, no. 07 (2010): 961-974.
\bibitem{BKS}
H. U. Boden and H. Karimi and A. S. Sikora, \textit{Adequate links in thickened surfaces and the generalized Tait conjectures}. arXiv preprint arXiv:2008.09895 (2020).
\bibitem{Buck}
A. Barbensi and D. Buck and H. Harrington and M. Lackenby, 
\textit{Double branched covers of knotoids}, arXiv:1811.09121 math.GT (2018), (to appear in Communications in Analysis and Geometry).
\bibitem{Bar}
A. Barbensi and D. Goundaroulis, \textit{$f$-distance of knotoids and protein structure}, Proceedings of the Royal Society A 477, no. 2246 (2021): 20200898.
\bibitem{CSK}
S. Carter and S. Kamada and M. Saito, \textit{Stable equivalence of knots on surfaces and virtual knot cobordisms}. Journal of Knot Theory and Its Ramifications 11, no. 3 (2002): 311-322
\bibitem{Dab} P. Dabrowski-Tumanski and D. Goundaroulis and A. Stasiak and J. Sulkowska, \textit{$\theta$-curves in proteins}, arXiv:1908.05919 (2019).
\bibitem{Deng} Q. Deng and L. H. Kauffman, \textit{The generalized Yamada polynomials of virtual spatial graphs}, Topology and its Applications 256 (2019): 136-158.
\bibitem{Dim}
D. Goundaroulis and J. Dorier and F. Benedetti and A. Stasiak, \textit{Studies of global and local entanglements of individual protein chains using the concept of knotoids}, Scientific reports 7, no. 1 (2017): 1-9.
\bibitem{Flem} T. Fleming and B. Mellor, \textit{Virtual spatial graphs}, Kobe journal of mathematics 24 (2007): 67-85.
\bibitem{GK} 
N. Gügümcü and L. H. Kauffman, \textit{New invariants of knotoids}, European Journal of Combinatorics 65 (2017): 186-229.
\bibitem{GK2}
N. Gügümcü and L. H. Kauffman, \textit{Parity, virtual closure and minimality of knotoids}, Journal of Knot Theory and Its Ramifications 30, no. 11 (2021): 2150076.

\bibitem{Graphoid} 
N. Gügümcü and B. Gabrovsek and L. H. Kauffman, \textit{Invariants of Bonded Knotoids and Applications to Protein Folding}, Symmetry 14, no. 8 (2022): 1724.



\bibitem{K} L. H. Kauffman,
\textit{State Models and the Jones Polynomial},
{\em Topology} {\bf 26} (1987), 395--407.

\bibitem{KGraph}  L. H. Kauffman, \textit{Invariants of graphs in three-space}, Trans. Amer. Math. Soc. 311 (1989), no. 2, 697–710.
\bibitem{KVirtual}
L. H. Kauffman, \textit{Virtual Knot Theory}, European Journal of Combinatorics, 20, (1999), 663-690

\bibitem{KExtended}
L. H. Kauffman, \textit{An extended bracket polynomial for virtual knots and links}, Journal of Knot Theory and Its Ramifications 18, no. 10 (2009): 1369-1422.
\bibitem{KodLam}
D.Kodokostas, S.Lambropoulou, \textit{Rail knotoids}, Journal of Knot Theory and Its RamificationsVol. 28, No. 13, 1940019 (2019)

\bibitem{LK} 
W. R. Lickorish and M. B. Thistlethwaite, \textit{Some links with non-trivial polynomials and their crossing-numbers}, Commentarii Mathematici Helvetici 63, no. 1 (1988): 527-539.
\bibitem{MM}
C. McPhail-Snyder and K. Miller, \textit{Planar diagrams for local invariants of graphs in surfaces}, Journal of Knot Theory and Its Ramifications 29, no. 01 (2020): 1950093.
\bibitem{Motohashi}
T. Motohashi and O. Yashiyuki and Kouki Taniyama, \textit{Yamada polynomial and crossing number of spatial graphs}, Rev. Mat. Univ. Complut. Madrid 7, no. 2 (1994): 247-277.
\bibitem{Simon}

J. Simon, \textit{Topological chirality of certain molecules}, Topology 25, no. 2 (1986): 229-235.
\bibitem{Turaev} 
V. Turaev, \textit{Knotoids}, Osaka J. Math. 49 (2012), no. 1, 195–223. 


\bibitem{WITT} E. Witten. \textit{Quantum Field Theory and the Jones Polynomial}, Comm. in Math. Phys.
Vol. 121 (1989), 351-399. \bigbreak 


\end{thebibliography}
\end{document}